\documentclass{amsart}
\usepackage{amssymb,hyperref}

\newcommand{\al}{\alpha}\newcommand{\be}{\beta}
\newcommand{\de}{\delta}
\newcommand{\ep}{\epsilon}
\newcommand{\la}{\lambda}

\newcommand{\om}{\omega}\newcommand{\Om}{\Omega}

\def\<{\langle}
\def\>{\rangle}
\newcommand{\R}{\mathbb{R}}\newcommand{\Z}{\mathbb{Z}}
\newcommand{\N}{\mathbb{N}}

\def\CO{\mathcal {O}}

\newcommand{\pt}{\partial_t}\newcommand{\pa}{\partial}

\newcommand{\les}{{\lesssim}}

\newcommand{\beeq}{\begin{equation}}\newcommand{\eneq}{\end{equation}}

\newtheorem{thm}{Theorem}[section]
\newtheorem{prop}[thm]{Proposition}
\newtheorem{coro}[thm]{Corollary}
\newtheorem{rem}{Remark}[section]

\newtheorem{lem}[thm]{Lemma}

\newenvironment{prf}{\noindent {\bf Proof.} }{\endprf\par}
\def \endprf{\hfill  {\vrule height6pt width6pt depth0pt}\medskip}

\numberwithin{equation}{section}

\begin{document}

\title[Strauss conjecture on asymptotically Euclidean manifolds]{Concerning the Strauss conjecture on asymptotically Euclidean manifolds}

\thanks{The first author was supported in part by NSFC 10871175 and
10911120383.}

\author{Chengbo Wang}
\address{Department of Mathematics, Johns Hopkins University, Baltimore,
Maryland 21218} \email{wangcbo@jhu.edu}


\author{Xin Yu}
\address{Department of Mathematics, Johns Hopkins University, Baltimore,
Maryland 21218} \email{yuxin@jhu.edu}


\dedicatory{} \commby{}

\begin{abstract}
 In this paper we verify the Strauss conjecture for semilinear wave
 equations on asymptotically Euclidean manifolds when $n=3,4$. We also give an almost sharp lifespan for the subcritical case $2\le p<p_c$ when $n=3$. The main ingredients include a Keel-Smith-Sogge type estimate with $0<\mu<1/2$ and weighted Strichartz estimates of order two.
\end{abstract}

\maketitle

\tableofcontents

\section{Introduction and Main Results}
This paper is devoted to the study of the semilinear wave equation
on asymptotically Euclidean non-trapping Riemannian manifolds with small initial data. In particular, we verify the Strauss conjecture in this setting when $n=3, 4$ and $p> p_c$. Moreover,
we obtain an almost sharp lifespan for the solution when $2\le p< p_c$ and $n=3$.

In the Minkowski space-time, this problem has been
thoroughly studied. The work on global existence part (i.e.
$p>p_c$) is initiated by John \cite{John} for $n=3$ and ended by
Georgiev, Lindblad and Sogge \cite{GLS} and Tataru
\cite{Ta01}. It is known that $p>p_c$ is necessary for
global existence, even with small data, see \cite{Sideris}, \cite{YorZh06}, \cite{Zhou07} and reference therein. Moreover, when $n=3$ and $p\le p_c$, the sharp lifespan is known in Zhou \cite{Zhou1} (see also \cite{LdSo96} for lower bound of the lifespan $p\le p_c$ and $n\ge 3$, and \cite{ZhouHan} for upper bound of the lifespan when $p<p_c$ and $n\ge 3$).

When dealing with semilinear wave equations, we know
that the Keel-Smith-Sogge (KSS) estimate plays an important role, which is originated by
Keel, Smith and Sogge \cite{KSS} and states that \beeq
\label{OKSS} (\log (2+T))^{-1/2}\|\<x\>^{-1/2}u'\|_{L^2([0,T]\times
\R^3)}\les\|u'(0,\cdot)\|_{L^2(\R^3)}+\int_0^T\|F(s,\cdot)\|_{L^2(\R^3)}\;ds,
\eneq where $u$ solves the equation $\Box\, u = F$ and $u'=(\pa_t u,
\pa_x u)$. This estimate has been generalized for general weight of form $\<x\>^{-a}$ with $a\geq 0$ (see \cite{JWY} and references therein).

Recently, Bony and H\"{a}fner \cite{BoHa} obtained a weaker
version of the KSS estimates for asymptotically Euclidean space
when the metric is non-trapping. With this estimate, they were able
to show the global and long time existence for quadratic
semilinear wave equations with dimension $n\ge 4$ and $n=3$. Then
Sogge and Wang \cite{SW} proved the almost global existence for
$3$-D quadratic semilinear equations by obtaining the sharp KSS
estimates for $a=1/2$. Together with the KSS estimates, they also
proved the Strauss conjecture for $n=3$ and $p>p_c$ with
spherically symmetric metric. The proof is based on weighted
Strichartz estimates, and it is the weighted Strichartz estimates of
higher order where the additional symmetric assumption is posed to
avoid the technical difficulties when commutating the Laplacian
with the vector fields.

In this work, we are able to overcome the difficulties of
commutating vector fields and verify the weighted Stricharz
estimates and energy estimates with derivatives up to second
order, for a general metric. This enables us to prove the Strauss
conjecture with $p>p_c$ for $n=3,4$. Moreover, we are able to get
the KSS estimates for $0<a<1/2$, by applying the corresponding
estimates for wave equations with variable coefficients
(see \cite{MeSo06_01}, \cite{HiWaYo}). With these estimates in hand,
we can also prove the local existence for $2\leq p< p_c$ when
$n=3$ with almost sharp lifespan.

Let us now state our results precisely. First, we introduce the necessary notations. We consider asymptotically
Euclidean manifolds $( \R^n , g)$ with $n \geq 3$ and
\begin{equation*}
g = \sum_{i,j=1}^{n} g_{ij} (x) \, d x^i \, d x^j .
\end{equation*}
We suppose $g_{ij} (x) \in C^{\infty} ( \R^{n} )$ and, for some
$\rho
>0$,
\begin{equation}\tag{H1} \label{H1}
\forall \alpha \in \N^n \qquad \partial^{\alpha}_x ( g_{ij} -
\delta_{ij} ) = \CO ( \< x \>^{- \vert \alpha \vert - \rho} ) ,
\end{equation}
with $\delta_{ij}=\delta^{ij}$ being the Kronecker delta function.
We also assume that
\begin{equation}\tag{H2} \label{H2}
g \text{ is non-trapping.}
\end{equation}
Let $g (x) = ( \det ( g ) )^{1/4}$. The Laplace--Beltrami
operator associated with $g$ is given by
\begin{equation*}
\Delta_{g} = \sum_{ij} \frac{1}{g^2} \partial_i g^{ij} g^2 \partial_j ,
\end{equation*}
where $g^{ij} (x)$ denotes the inverse matrix of $g_{ij}(x)$. It is easy to see that
$-\Delta_{g}$ is a self-adjoint non-negative operator on
$L^2(\R^n, g^2 d x)$, while $P=-g\Delta_{g}g^{-1}$ is a
self-adjoint non-negative operator on $L^2(\R^n, d x)$.


Let $p>1$,
$$s_c=\frac{n}{2}-\frac{2}{p-1},\quad
s_{d}=\frac{1}{2}-\frac{1}{p} $$
and
$p_c$ be the
positive root for
$$(n-1)p^2-(n+1)p-2=0.$$ Note that $p_c=1+\sqrt{2}$ for $n=3$ and $p_c=2$ for $n=4$. The semilinear wave equations we will
consider are
\begin{equation}\label{eq}
\begin{cases}
(\partial_t^2-\Delta_g)u(t,x)=F_p(u(t,x)), \quad (t,x)\in \R_+\times \R^n
\\
u(0,x)=u_0(x), \quad \partial_t u(0,x)=u_1(x), \quad x\in \R^n.
\end{cases}
\end{equation}
  We will assume that the nonlinear term behaves like
$|u|^p$, 
and so we assume that
\begin{equation}\label{Fp}
\sum_{0\le j\le 2} |u|^j\, |\, \partial^j_u F_p(u)\, | \, \lesssim
\, |u|^p, \text{ for $|u|$ small.}
\end{equation}
Finally we introduce the notation for vector fields $Z=\{\pa_x,
\Om_{ij}: 1\le i\le j\le 3\}$, $\Gamma=\{\pt\}\cup Z$, where
$\Om_{ij}=x_i\pa_j-x_j\pa_i$ is the rotational vector field, and
define $\tilde\pa_i=\pa_i g^{-1}, \tilde\Om_{ij}=\Om_{ij} g^{-1}$.

Now we can state our main results.
\begin{thm}
\label{Strauss} Suppose \eqref{H1} and \eqref{H2} hold with $\rho >
2$, $n=3,4$, and $p_c<p<1+4/(n-1)$. Then for any $\ep>0$
such that $($recall that $s_c>s_d$ since $p>p_c)$
\begin{equation}\label{s_1}
  s=s_c-\ep\in \left(s_d, \frac 12\right)
\end{equation}
 there is a $\de>0$ depending on $p$ so that
\eqref{eq} has a global solution satisfying
$(Z^\alpha u(t,\cdot), \partial_t Z^\alpha u(t,\cdot))\in
\dot{H}^s \times \dot{H}^{s-1}$, $|\alpha|\le 2$, $t\in \R_+$,
whenever the initial data satisfies
\beeq\label{70-eqn-data2}
\sum_{|\alpha|\le 2}\left(\, \|Z^\alpha u_0\|_{\dot H^s} +\|Z^\alpha
u_1\|_{ \dot H^{s-1}}\, \right)<\de .
\eneq
Moreover, in the case $n=3$, we can relax the assumption for $\rho$
to $\rho>1$. More precisely, if $F_p(u)$ satisfies
\beeq\label{assum}\sum_{0\le j\le 1} |u|^j\, |\, \partial^j_u F_p(u)\, | \, \lesssim
\, |u|^p\eneq
instead of \eqref{Fp}, for any $p_c<p<3$ and any $\ep>0$
such that \eqref{s_1} is true, the problem \eqref{eq} has a global
solution satisfying $(Z^\alpha u(t,\cdot), \partial_t Z^\alpha
u(t,\cdot))\in \dot{H}^s \times \dot{H}^{s-1}$, $|\alpha|\le 1$,
$t\in \R_+$, whenever the initial data satisfies
\begin{equation}\label{70-eqn-SLW-data}
\sum_{|\alpha|\le 1}\left(\, \|Z^\alpha u_0\|_{\dot H^s} +\|Z^\alpha
u_1\|_{ \dot H^{s-1}}\, \right)<\delta.
\end{equation}
\end{thm}

We also have the following existence result for
$2\le p< p_c$ when $n=3$, where the lifespan is almost sharp (see \cite{Zhou1} for the blow up results).
\begin{thm}
\label{Strauss2} Suppose \eqref{H1} and \eqref{H2} hold with $\rho
> 2$, $n=3$, and $2\le p<p_c=1+\sqrt 2$.
Then there exists $c>0$ and $\delta_0>0$ depending on $p$ so that
\eqref{eq} has a solution in $[0,T_\de]\times \R^3$ satisfying
$(Z^\alpha u(t,\cdot), \partial_t Z^\alpha u(t,\cdot))\in \dot{H}^s
\times \dot{H}^{s-1}$, $|\alpha|\le 2$, $t\in [0, T_\de]$, with
\beeq\label{s_2}
s=s_d,\ T_\de=c\ \de^{\frac{p(p-1)}{p^2-2p-1}+\ep},
\eneq
 whenever
$\ep>0$ and the initial data satisfies \eqref{70-eqn-SLW-data} with $\de<\de_0$. Moreover, we can relax the assumption for $\rho$
to $\rho>1$, when $F$ satisfies \eqref{assum} and $s=s_d+\ep'$ for some small $\ep'>0$.
\end{thm}
\begin{rem}
The above result for $p<p_c$ is a natural extension of Theorem 4.1 in Chapter 4 of
Sogge \cite{So2} and Theorem 4.2 of Hidano \cite{H}. See also
Theorem 4.1 of Yu \cite{Yu09} and Theorem 6.1 of \cite{JWY} for
closely related ${\dot H}^{s_d}$-results.\end{rem}

For convenience we define the norm $Y_{s,\ep}$ as
$$\|f(x)\|_{Y_{s,\ep}}=\|\<x\>^{{-(1/2)-s-\ep}}f(x)\|_{L^2_x}\ .$$
The main estimate we will need to prove Theorem \ref{Strauss} is as
follows.
\begin{thm}\sl \label{mainest}
Let $u$ be the solution of the linear equation
\begin{equation} \label{LW}
\begin{cases}
(\partial_t^2+P)u(t,x)=F(t,x), \quad (t,x)\in \R_+\times \R^n
\\
u(0,x)=u_0(x), \quad \partial_t u(0,x)=u_1(x), \quad x\in \R^n
\end{cases}
\end{equation}
with $F=0$. Assume that \eqref{H1} and \eqref{H2} hold with $\rho >
2$, $n\ge 3$, $2< p\le \infty$ and $s\in (s_d, 1)$. For all
$\ep>0$ and $\eta>0$
  small enough, we have
\begin{equation}
\label{highorderest}\sum_{|\al|\le 2}\|Z^\al u \|_{L^2_t
Y_{s,\ep}}+\||x|^{n/2-(n+1)/p-s-\ep}Z^\al u\|_{L^p_t
L^p_{|x|} L^{2+\eta}_{\om}(\{|x|>1\})}\les
 \sum_{|\al|\le 2} \left(\Vert Z^\al u_0\Vert_{{\dot H}^{s}}+
 \Vert Z^\al u_1\Vert_{ \dot H^{s-1}}\right) \ ,\end{equation}
 and for $s\in[0,1]$,
 \begin{equation}
\label{highorderenergyest}\sum_{|\al|\le 2}\left ( \|Z^\al u
\|_{L^\infty_t \dot H^s}+\|\pt Z^\al u\|_{L^\infty_t \dot
H^{s-1}}+\|Z^\al u\|_{L^p_t L^{q_s}_x(|x|\leq 1)}\right) \les
 \sum_{|\al|\le 2} \left(\Vert Z^\al u_0\Vert_{{\dot H}^{s}}+
 \Vert Z^\al u_1\Vert_{ \dot H^{s-1}}\right) \ ,
 \end{equation}
 where $q_s=2n/(n-2s)$. 
 On the other hand, if we assume $\rho>1$ instead of $\rho>2$, we have
 the same estimates of first order $(|\al|\le 1)$.
\end{thm}

Here, the angular mixed-norm space $L^p_{|x|}L^r_\omega$ is defined
as follows
$$\|f\|_{L^p_{|x|}L^r_\om({\mathbb R}^n)}=
\left(\, \int_0^\infty \, \Bigl( \, \int_{\mathbb{S}^{n-1}}
|f(\la \omega)|^r \, d\omega \, \Bigr)^{p/r}\, \la^{n-1} d\la\,
\right)^{1/p}\ ,$$ which is consistent with the usual Lebesgue space
$L^p_x$ when $p=r$.

Recall that Theorem \ref{mainest}, with order $0$ ($|\al|=0$) and
$\rho>0$, has been proved in Theorem 1.6 of \cite{SW} for any
$s\in (s_d, 1]$ in general. However, the estimates with higher
order derivatives are much more complicated. As we will see, one
of the main difficulties in the proof is that we need to establish
the relation between $P$ and the vector fields $Z$, where only the
powers of $P$ can be commutated with the equation $\pt^2+P$. The
most difficult part of the commutators comes from the commutator
of $P$ and the rotational vector fields $\Omega_{ij}$. Another difficulty arises from the estimates with second order derivatives, and the techniques we use here will require the assumption $\rho>2$ instead of $\rho>1$.

To obtain Theorem \ref{Strauss2} we will need the following local
in time weighted Strichartz estimates .
\begin{thm}\sl \label{mainest2}
Let $u$ be the solution of \eqref{LW} with $F=0$. Assume that
\eqref{H1} and \eqref{H2} hold with $\rho > 2$, $n\ge 3$,
$0<a<1/p$, $2\leq p< \infty$ and $s=s_d$. Then we have \beeq
\label{highorderest2}\sum_{|\al|\le 2}
\|\<x\>^{-a}|x|^{(n-1)s}Z^\al u \|_{L^p_t
L^p_{|x|}L^2_{\omega}([0,T]\times \R^n)}\les (1+T)^{(1/p)-a+\ep}\sum_{|\al|\le 2} \left  (\Vert Z^\al u_0\Vert_{{\dot
H}^{s}}+ \| Z^\al u_1\|_{ \dot H^{s-1}} \right ). \eneq
 On the other hand, if we assume $\rho>1$ instead of $\rho>2$, we have
 the same estimates of first order $(|\al|\le 1)$, with $s=s_d+\ep'$ for small enough $\ep'>0$.
\end{thm}

\begin{rem}
Note that the estimates in the above two theorems are given for
solutions of $(\pt^2+P)u=F$, which has the benefit that the solution can
be represented by the following formula $$u(t)=\cos (t P^{1/2})u_0+P^{-1/2}\sin (tP^{1/2})
u_1+\int_0^t P^{-1/2}\sin ((t-s)P^{1/2}) F(s)ds\ . $$ All of the operators occurring in this formula commutates with the wave operator $\pt^2+P$. In general, an
estimate for $-\Delta_g$ will corresponds another estimate for $P$.
For example, if we have the estimate \eqref{highorderest} for $P$,
consider the equation
\begin{equation} \label{70-eqn-LW-transf}
\begin{cases}
(\partial_t^2-\Delta_{g})v(t,x)=G(t,x), \quad (t,x)\in
\R_+\times \R^n
\\
u(0,x)=v_0(x), \quad \partial_t v(0,x)=v_1(x), \quad x\in \R^n.
\end{cases}
\end{equation}
Notice that if we let $u=g v$ and $F=g G$, then
\beeq\label{70-est-EquiEqn}(\pt^2-\Delta_{g})v=G \Leftrightarrow
(\pt^2+P)u=F.\eneq Thus we have also the estimate
\eqref{highorderest} for $-\Delta_g$.
\end{rem}

The paper is arranged as follows. In Section 2 we prove the weighted Stricharz
estimates and energy estimates (i.e. Theorem
\ref{mainest}); In Section 3 we prove higher order KSS estimates and local in time weighted Strichartz estimates (i.e. Theorem \ref{mainest2}); Finally in Section 4 we will see
how Theorem \ref{mainest} and Theorem \ref{mainest2} imply the
Strauss conjecture when $n=3,4$.


\section{Weighted Strichartz and Energy Estimates}
In this section, we will give the proof of our main estimates
\eqref{highorderest} and \eqref{highorderenergyest}.

In what follows, ``remainder terms", $r_{j}$, $j\in \N$, will
denote any smooth functions such that \begin{equation} \label{H3}
\partial^{\alpha}_{x} r_{j} (x) = O \big( \< x \>^{-\rho - j - \vert \alpha \vert} \big) ,\
\forall\al \ ,
\end{equation}
thus $P=-g\Delta_g g^{-1}=-\Delta+r_0\pa^2+r_1\pa+r_2$.

\subsection{Preparation}

Before we go through the proof of the main theorems, we will present
several useful lemmas. The first one is the KSS estimates
(Keel-Smith-Sogge estimates) on asymptotially Euclidean manifolds
obtained in \cite{BoHa} and \cite{SW}, and the second one gives the
relation between the operators $P^{1/2}$ and $\pa_x$.

\begin{lem}[KSS estimates]\sl \label{lem00}
Assume that \eqref{H1} and \eqref{H2} hold with $\rho > 1$.  Let $N
\ge 0$, $ \mu\ge 1/2$ and
\begin{equation*}
A_{\mu} (T) = \left\{ \begin{aligned}
&(\log (2+T))^{-1/2} &&\mu = 1/2 ,   \\
&1 &&\mu > 1/2 .
\end{aligned} \right.
\end{equation*}
Then the solution of \eqref{LW} satisfies
\begin{align}\label{KSS}
\sup_{0\le t\le T } \sum_{1 \leq k + j \leq N+1} \big\Vert
\partial_{t}^{k} P^{j/2}g u (t & , \cdot ) \big\Vert_{L^2_x} +
\sum_{\vert \alpha \vert \leq N} A_\mu(T)
 \big\Vert \<x\>^{-\mu} \left(|(\Gamma^{\alpha} u)'|+  \frac{|\Gamma^{\alpha} u|}{\<x\>}\right)
 \big\Vert_{L^2_T L^2_x} \nonumber \\
& \lesssim \sum_{\vert \alpha \vert \leq N} \big\Vert (Z^{\alpha}u)' (0, \cdot ) \big\Vert_{L^2_x}
 + \sum_{\vert \alpha \vert \leq N}\big\Vert \Gamma^{\alpha} F(s, \cdot ) \big\Vert_{L^1_T
 L^2_x}\ ,
\end{align} where $L^q_T L^r_x=L^q([0,T]; L^r(\R^n))$.
\end{lem}
\begin{proof}
This is Theorem 1.3 in \cite{SW}.
\end{proof}

\begin{rem}\label{rem}
Here, we notice that the estimate \eqref{KSS} still holds if we replace $\Gamma$ and $Z$ with $\pa_{x}$ in
\eqref{KSS} $($see $(3.6)$ in \cite{SW}$)$. Moreover, we will see later
in Proposition \ref{Prop3.2} that the corresponding estimates for
$0<\mu<1/2$ also hold.
\end{rem}

The next lemma gives the relation between the operators $\pa_x$
and $P^{1/2}$.
\begin{lem}\label{lem0}
If $s\in[-1,1]$, then \beeq\label{d}\|u\|_{\dot{H}^s}\simeq
\|P^{s/2}u\|_{L^2_x}.\eneq If $s\in[0,1]$,
\beeq\label{f}\|\tilde{\pa}_j u\|_{\dot{H}^{-s}} \les \|P^{1/2}
u\|_{\dot{H}^{-s}},\eneq \beeq\label{e}\|P^{1/2} u\|_{\dot{H}^s}\les
\sum_j\|\tilde{\pa}_j u\|_{\dot{H}^s}.\eneq Moreover, we have for
$s\in (0,2]$ and $1<q <n/s$, \beeq\label{g}\|P^{s/2}u\|_{L^q_x}\les
\|u\|_{\dot{H}^{s,q}}. \eneq
\end{lem}

\begin{proof}
This is just Lemma 2.4 in \cite{SW}.
\end{proof}

The following three lemmas are proved to deal with the commutator
terms we will encounter in the proof of our higher order estimates
\eqref{highorderest} and \eqref{highorderenergyest}.

\begin{lem}
\label{lem1} Let $u$ solve the wave equation \eqref{LW}. Then for
any $s\in[0,1]$ and $\ep>0$, we have:
\beeq\label{70-est-Morawetz-LowOrder2}\|u\|_{L^2_t
Y_{s,\ep}}\lesssim \|u_0\|_{\dot H^s}+\|u_1\|_{\dot
H^{s-1}}+\|\<x\>^{(1/2)+\ep}F\|_{L^2_t \dot{H}^{s-1}}\eneq
\end{lem}

\begin{proof} We give first the proof in the case $u_0=u_1=0$.
First, from Remark 2.1 in \cite{SW} we know
\beeq\label{est1}
\big\Vert \< x \>^{-(3/2)-\ep} u
  \big\Vert_{L^2 ( \R \times \R^{n} )} \lesssim
    \Vert \< x \>^{(1/2)+\ep} F\Vert_{L^2 ( \R \times \R^{n} )}.\eneq
Next, using the KSS estimates on asymptotically Euclidean manifolds (Lemma 2.1 in \cite{SW}) together
with \eqref{est1}, we have
    \begin{eqnarray} \label{est2} \big\Vert \< x \>^{-(1/2)-\ep} u
  \big\Vert_{L^2_t\dot{H}^1 ( \R \times \R^{n} ) }&\lesssim&
\big\Vert \< x \>^{-(3/2)-\ep} u \big\Vert_{L^2 ( \R \times
\R^{n} )}+\big\Vert \< x \>^{-(1/2)-\ep} u'
  \big\Vert_{L^2 ( \R \times \R^{n} )}\nonumber \\
  &\lesssim&   \Vert \< x \>^{(1/2)+\ep} F\Vert_{L^2 ( \R \times \R^{n} )}.\end{eqnarray}
Since $P$ is self adjoint, for any fixed $T>0$, if we let $\Box_P v=(\pa_t^2+P)v=G$ with vanishing initial data at $T$, then
\begin{eqnarray}
 \big\Vert \< x \>^{-(1/2)-\ep} u
\big\Vert_{L^2 ( [0,T] \times \R^{n} )}&=& \sup_{\|\<x\>^{1/2+\ep}G\|_{L^2([0,T]\times\R^n)}\le 1}\<u,G\>\nonumber\\
&=& \sup_{\|\<x\>^{1/2+\ep}G\|_{L^2([0,T]\times\R^n)}\le 1}\<\Box_P u,v\>\nonumber \\
&\lesssim&
    \Vert \< x \>^{(1/2)+\ep} \Box_P u\Vert_{L^2_t\dot H^{-1}( [0,T] \times \R^{n} )}\|\<x\>^{-(1/2)-\ep}v\|_{L^2_t\dot H^1([0,T]\times\R^n)}\nonumber\\
    &\les&\Vert \< x \>^{(1/2)+\ep} F\Vert_{L^2_t\dot H^{-1}( [0,T] \times \R^{n} )}\|\<x\>^{(1/2)+\ep}G\|_{L^2([0,T]\times\R^n)} \nonumber\\
    &\les& \Vert \< x \>^{(1/2)+\ep} F\Vert_{L^2_t\dot H^{-1}( \R \times \R^{n} )}.\nonumber\end{eqnarray}
Since the constants in the inequality are independent with $T$, we get
\beeq\label{est3}
\big\Vert \< x \>^{-(1/2)-\ep} u
\big\Vert_{L^2 ( \R \times \R^{n} )}\les \Vert \< x \>^{(1/2)+\ep} F\Vert_{L^2_t\dot H^{-1}( \R \times \R^{n} )}.
\eneq
Now we can get the desired estimate
\eqref{70-est-Morawetz-LowOrder2} for $u_0=u_1=0$ by an
interpolation between \eqref{est1} and \eqref{est3}. The estimate
with $F=0$ follows just from the estimate
\eqref{highorderest} of order $0$, which is proved in \cite{SW}.
\end{proof}

\begin{lem}
\label{lem0.1} Let $w$ solve the wave equation \eqref{LW} with
$u_0=u_1=0$. Then for $s\in[0,1]$ and $\ep>0$, \beeq\label{est0}
\|w\|_{L_t^\infty\dot{H}_x^s}\les
\|\<x\>^{1/2+\ep}F\|_{L_t^2\dot{H}^{s-1}_x}. \eneq
\end{lem}
\begin{proof}
We will show this estimate by interpolation. For $s=1$, notice
that KSS estimates in Lemma \ref{lem00} give us
$$\|\<x\>^{{-1/2-\ep}}e^{itP^{1/2}}f\|_{L_{t,x}^2}\les \|f\|_{L^2_x}\ .$$
After the standard $TT^{*}$ argument, we get
$$\left\|\int_\R e^{-isP^{{1/2}}}G(s,\cdot) ds\right\|_{L_x^2}\les\|\<x\>^{{1/2}+\ep}G(t,x)\|_{L^2_{t,x}}\ ,$$
and so
\begin{eqnarray*}
\left\|\int_\R e^{i(t-s)P^{1/2}}F(s) d s\right\|_{L_t^\infty L^2_x}&\les&
\left\|\int_\R e^{-isP^{1/2}}F(s)d s\right\|_{L^2_x}\\
&\les&\|\<x\>^{(1/2)+\ep}F\|_{L^2_{t,x}}\ .
\end{eqnarray*}
Thus by the Christ-Kiselev lemma (cf. \cite{CK}) we have
$$
\left\|\int_0^t e^{i(t-s)P^{1/2}}F(s)ds\right\|_{L_t^\infty L^2_x}\les
\|\<x\>^{(1/2)+\ep}F\|_{L^2_{t,x}}\ .
$$
Recall that $w= P^{-1/2} \int_0^t \sin ((t-s)P^{1/2}) F(s) d
s$, then we get the proof of \eqref{est0} for the case $s=1$ as
follows,
$$
\|w\|_{L_t^\infty\dot{H}^1}\approx
\|P^{1/2}w\|_{L^\infty_tL^2_x}\les
\|\<x\>^{(1/2)+\ep}F\|_{L^2_{t,x}}\ .$$ For $s=0$, by \eqref{KSS},
$$\|\<x\>^{-(1/2)-\ep} w \|_{L_t^2\dot H^1}\les \|\<x\>^{-(3/2)-\ep}w\|_{L_{t,x}^2}+\|\<x\>^{-(1/2)-\ep}\pa_x w\|_{L^2_{t,x}}
\les \| F\|_{L^1_t L^2_x}\ . $$
The above inequality, combined with a similar duality argument for \eqref{est3}, gives
$$\|w\|_{L_t^\infty L_x^2}\les\|\<x\>^{(1/2)+\ep }F\|_{L_t^2\dot H^{-1}}\ ,$$
which is just the estimate for $s=0$. This completes the proof if
we interpolate between the estimates for $s=0$ and $s=1$.
\end{proof}

On the basis of the above two lemmas, we can control the commutator terms
by a kind of weighted $L^2_t\dot H_x^{s-1}$ norm. Then with the
following lemma we will be able to bound this norm by the good
terms, thus we can use the argument as in \cite{SW} to get over
the difficulty on error terms.
\begin{lem}
\label{lem0.2} Let $n\ge 3 $, $N\ge 1$ and $u$ be the solution to
\eqref{LW} with $F=0$. Then for any $s\in[0,1]$, $\ep>0$ and
$|\al|=N$, we have
\beeq \label{forlem5-2}\sum_{|\al|=
N}\|\<x\>^{-(1/2)-\ep}\pa_x^\al u\|_{L^2_t\dot
H^{s-1}}\les\|u_0\|_{\dot H^{N+s-1}\cap \dot H^s}+\|u_1\|_{\dot
H^{N+s-2}\cap \dot H^{s-1}}\ .\eneq
\end{lem}
\begin{proof}
The estimate for $s=1$ follows directly from the KSS estimates
\eqref{KSS} and Remark \ref{rem}. Moreover, we have the following
estimate \begin{multline}\label{forlem5-3}\| \<x\>^{-(1/2)-\ep} u \|_{L^2_t
L^2_x}=\| \<x\>^{-(1/2)-\ep} P^{1/2}(P^{-1/2}u) \|_{L^2_t
L^2_x}\\\les\|P^{-1/2}u_0\|_{\dot H^1}+\|P^{-1/2}u_1\|_{L^2_x}\les \|u_0\|_{L^2_x}+\|u_1\|_{\dot H^{-1}}\ .\end{multline}

 For $s=0$, first notice that since $n\geq 3$, we
have Hardy's inequality
$$\|\left<x\right>^{-2} x h\|_{L^2_x}\les \|h\|_{\dot H^1}\ ,$$
and the duality gives
$$\|\left<x\right>^{-2} x f\|_{\dot H^{-1}}\les \|f\|_{L^2_x}\ .$$
Using the above estimate together with the KSS estimates and \eqref{forlem5-3}, we get
\begin{eqnarray*}\|\<x\>^{-(1/2)-\ep}\pa_x^\al u\|_{L^2_t\dot
H^{-1}}&\les& \|\<x\>^{-(5/2)-\ep}x\pa^{\al-1}_x u\|_{L^2_t\dot
H^{-1}}+\|\<x\>^{-(1/2)-\ep}\pa_x^{\al-1}
u\|_{L^2_tL^2_x}\\&\les&\|\<x\>^{-(1/2)-\ep}\pa_x^{\al-1}
u\|_{L^2_tL^2_x}\\
&\les& \|u_0\|_{\dot H^{N-1}}+\|u_1\|_{\dot H^{N-2}\cap \dot H^{-1}}\ .
\end{eqnarray*}
Now \eqref{forlem5-2} follows from an interpolation between $s=0$
and $s=1$.
\end{proof}

Next we give three lemmas that will be used to prove the second
order part of the estimates \eqref{highorderest} and
\eqref{highorderenergyest}.

\begin{lem}  \label{lem7}
For $0 < \mu \leq 3/2$ and $k \geq 2$, we have
\begin{equation} \label{W7}
\big\Vert \< x \>^{- \mu} \widetilde{\partial}_{j_1} \cdots
\widetilde{\partial}_{j_k} u \big\Vert_{L^2_x} \lesssim
\sum_{j=0}^{[\frac{k-1}{2}]} \big\Vert \< x \>^{- \mu}
\widetilde{\partial} P^j u \big\Vert_{L^2_x} + \sum_{j=1}^{[\frac{k}{2}]}
\big\Vert \< x \>^{- \mu} P^j u \big\Vert_{L^2_x} \ ,
\end{equation} where $[a]$ denotes the integer part of $a$ $(\max\{k\in \Z, k\le a\})$.
\end{lem}
\begin{proof}
This is just Lemma 4.8 in \cite{BoHa}.
\end{proof}

\begin{lem}[Fractional Leibniz rule]\label{leibniz}
Let $0\le s<n/2$, $2\le p_i<\infty$ and $1/2=1/{p_i}+{1}/{q_i}\ (i=1,2)$. Then
$$\|fg\|_{\dot H^s}\les \|f\|_{L^{q_1}} \|g\|_{\dot H^{s,{p_1}}}+\|f\|_{\dot H^{s,p_2}}\|g\|_{L^{q_2}}\ .$$
Moreover, for any $s\in (-n/2, 0)\cup (0,n/2)$,
$$\|fg\|_{\dot H^s} \les \|f\|_{L^\infty\cap \dot H^{|s|,n/{|s|}}}\|g\|_{\dot H^s}\ .$$
\end{lem}
\begin{proof}
  The first inequality is well known, see, e.g., \cite{KG}. The second inequality with $s\ge 0$ is an easy consequence of the first inequality together with Sobolev embedding. Then the result for negative $s$ follows by duality.
\end{proof}

\begin{lem}
\label{lem6} For $f\in \dot H^s(\R^n)\cap \dot H^{s+2}(\R^n)$ with
$n\geq 3$ and $s\in [0,1]$, we have \beeq\label{2ndorderest}
\|\pa_x^2 f\|_{\dot H^s}\les \|Pf\|_{\dot H^s}+\|f\|_{\dot H^s}\ .
\eneq On the other
hand, 
\beeq\label{2ndorderest2} \|P f\|_{\dot
H^s}\les \sum_{|\al|\leq 2}\|\pa_x^\al f\|_{\dot H^s}\ . \eneq
\end{lem}

\begin{proof}
First, we give the proof for the estimate \eqref{2ndorderest2}.
When $s=0$, notice that $Pf=g^{ij}\pa_i\pa_j f+r_1\pa_x f+r_2f$, we
have
$$\|Pf\|_{L^2_x}\les \|\pa_x^2 f\|_{L^2_x}+\|\pa_x f\|_{L^2_x}+\|f\|_{L^2_x}\les \|f\|_{\dot H^2\cap L^2_x}\ .$$
When $s=1$, recalling that $\pa^j r_i=O(\<x\>^{-\rho-i-j})$, by Hardy's inequality,
\beeq\nonumber\|\pa_x(r_2f)\|_{L^2_x}\les \|\pa_x(r_2)f\|_{L^2_x}+\|r_2\pa_x f\|_{L^2_x}\les \|\pa_x f\|_{L^2_x}.\eneq
Thus
\begin{eqnarray*} \|Pf\|_{\dot
H^1_x}=\|\pa_x Pf\|_{L^2_x}&\leq& \|\pa_x(g^{ij}\pa_i \pa_j
f)\|_{L^2_x}+\|\pa_x(r_1\pa_x
f)\|_{L^2_x}+\|\pa_x(r_2f)\|_{L^2_x}\\
&\les& \|\pa_x^3 f\|_{L^2_x}+\|\pa_x f\|_{L^2_x}\\
&\les& \|f\|_{\dot H^3\cap\dot H^1}\ .
\end{eqnarray*}
Our estimate \eqref{2ndorderest2} is obtained by an interpolation
between the above two estimates on $P f$.

Now we turn to the proof of the estimate \eqref{2ndorderest}.
First, when $s=0$, by elliptic property of $P$, we have
\beeq\label{lem6est1} \|\pa_x^2 f\|_{L_x^2}\les
\|Pf\|_{L_x^2}+\|f\|_{L_x^2}\ . \eneq Second, for $s=1$, using
\eqref{lem6est1},
\begin{eqnarray*}
\|\pa_x^3 f\|_{L^2_x}&\les& \|P\pa_x f\|_{L^2_x}+\|\pa_x f\|_{L^2_x}\\
&\les&\|[P,\pa_x]f\|_{L^2_x}+\|\pa_x Pf\|_{L^2_x}+\|\pa_x f\|_{L^2_x}\\
&\les&\|\sum_{|\al|\leq 2}r_{3-|\al|}\pa_x^\al
f\|_{L^2_x}+\|Pf\|_{\dot
H^1}+\|f\|_{\dot H^1}\\
&\les&\|Pf\|_{\dot H^1}+\|f\|_{\dot H^1}+\|f\|_{\dot H^2}\\
&\les&\|Pf\|_{\dot H^1}+\|f\|_{\dot H^1}+\ep\|f\|_{\dot H^3}+(1/\ep)\|f\|_{\dot H^1},\;\;\forall \ep>0.
\end{eqnarray*}
 Here we have used Hardy's inequality and the fact that $\dot H^3\cap\dot H^1\subset \dot H^2$. Now if we choose $\ep>0$ small enough and use
\eqref{2ndorderest2} with $s=0$, we have \beeq\label{lem6est2}
\|\pa_x^2 f\|_{\dot H^1}\les\|Pf\|_{\dot H^1}+\|f\|_{\dot H^1}\les
\|P P^{1/2}f\|_{L^2_x}+\|f\|_{\dot H^1}\les \|P^{1/2}f\|_{\dot
H^2}+\|P^{1/2}f\|_{L^2_x}\eneq  On the basis of \eqref{lem6est1} and
\eqref{lem6est2}, by an interpolation for the operator
$\pa^2P^{-1/2}$ and making use of Lemma \ref{lem0}, we have,
\begin{multline} \label{lem6est4}\|\pa_x^2f\|_{\dot H^s}\les
\|P^{1/2}f\|_{\dot H^{1+s}}+\|P^{1/2}f\|_{\dot
H^{s-1}}\\
\les\|P^{1/2}f\|_{\dot H^{1+s}}+\|P^{1/2+(s-1)/2}f\|_{L^2_x}\les\|P^{1/2}f\|_{\dot H^{1+s}}+\|f\|_{\dot H^s}.\end{multline}
We  need only to deal with the term $\|P^{1/2}f\|_{\dot
H^{1+s}}$. Note that for $s\in [0,1]$, we have
$$\|P^{-1/2}v\|_{\dot H^{1+s}}\les \|v\|_{\dot H^s}+\|v\|_{\dot
H^{-s}},
$$ which is true for $s=0$ (see \eqref{d}) and $s=1$ (see
\eqref{lem6est1}). Recalling that $P-g^{ij}\pa_i \pa_j=r_1
\pa_x+r_2$, and by Leibniz rule (see Lemma \ref{leibniz}), we have
for any small $0<\ep<\rho$,
\begin{eqnarray} \|P^{1/2}f\|_{\dot
H^{1+s}}&\les&
\|Pf\|_{\dot H^s}+\|Pf\|_{\dot H^{-s}}\nonumber\\
&\les&\|Pf\|_{\dot H^s}+\|f\|_{\dot H^{2-s}}+\|r_1 \pa_x f\|_{\dot H^{-s}}+\|r_2 f\|_{\dot H^{-s}}\nonumber\\
&\les&\|Pf\|_{\dot H^s}+\|f\|_{\dot H^{2-s}}+\|f\|_{\dot H^{1+\ep}}\nonumber\\
&\les& \|Pf\|_{\dot H^s}+\|f\|_{\dot H^{s}}^{\theta_1}\|f\|_{\dot
H^{2+s}}^{1-\theta_1}+\|f\|_{\dot H^{s}}^{\theta_2}\|f\|_{\dot
H^{2+s}}^{1-\theta_2},\;\;\; \text{where}\;
\theta_i\in (0,1].\nonumber\\
&\les& \|Pf\|_{\dot H^s}+\|f\|_{\dot H^{s}}^{\theta_1}\|\pa_x^2
f\|_{\dot H^{s}}^{1-\theta_1}+\|f\|_{\dot H^{s}}^{\theta_2}\|\pa^2_x
f\|_{\dot H^{s}}^{1-\theta_2}\ ,\label{lem6est3}\end{eqnarray} where
we have used the fact that $s\le 1+\ep, 2-s<2+s$ (so that
$\theta_i>0$) for $s\in (0, 1]$. Now our estimate \eqref{2ndorderest}
(for $s>0$) follows from \eqref{lem6est4} and \eqref{lem6est3}.
\end{proof}


\subsection{Proof of Theorem 1.3}

Now we are ready to give the proof of Theorem \ref{mainest}. Recall
that it has been proved in \cite{SW} that the result holds in the
case with order $0$ ($|\al|=0$) and $\rho>0$. Specifically, by KSS
estimates \eqref{KSS} and energy estimates, we have
\beeq\label{6.7}\|u\|_{L^2_t Y_{s,\ep}}+\|u\|_{L^\infty_t
\dot{H}^s}+\|\pt u\|_{L^\infty_t \dot{H}^{s-1}}\les \Vert
u_0\Vert_{\dot{H}^{s}}+\Vert u_1\Vert_{\dot H^{s-1}}\eneq for the
solution $u$ to the homogenous linear wave equation \eqref{LW} and
$s\in [0,1]$.

Recall that Fang and Wang obtained the following Sobolev
inequalities with angular regularity ((1.3) in \cite{FaWa})
\beeq\label{trace-g}\||x|^{n/2-s} f(x)\|_{L^\infty_{|x|}
L^{2+\eta}_\omega} \les \||x|^{n/2-s} f(x)\|_{L^\infty_{|x|}
H^{s-1/2}_\omega}
  \les \|  f\|_{\dot H^{s}} \eneq
for $s\in (1/2,n/2)$ and some $\eta>0$. By Lemma \ref{lem0}, we have
\begin{multline} \label{6.3} \||x|^{n/2-s} e^{i t P^{1/2}}
f(x)\|_{L^\infty_{t, |x|} L^{2+\eta}_\omega}\les \|
  e^{i t P^{1/2}} f(x)\|_{L^\infty_t\dot H^s_x}\\\les\|
  e^{i t P^{1/2}}P^{s/2} f(x)\|_{L^\infty_tL^2_x}\les
\|P^{s/2}f\|_{L^2_x}\les\|f\|_{\dot H^s}\end{multline} for $s\in (1/2, 1]$.

On the basis of KSS estimates, we can also obtain local energy decay
estimates
$$\|\phi u\|_{L^2_t H^s}\les \| u_0\Vert_{\dot{H}^{s}}+\Vert
u_1\Vert_{\dot H^{s-1}}$$ for $\phi\in C_0^\infty$ and $s\in
[0,1]$(see Lemma 2.6 in \cite{SW}). Then for any $p\geq 2$, \beeq
\label{6.8} \|\phi u\|_{L^p_t \dot{H}^s} \les \|\phi u\|_{L^2_t
\dot{H}^s}+\|\phi u\|_{L^\infty_t \dot{H}^s}\les  \Vert
u_0\Vert_{\dot{H}^{s}}+\Vert u_1\Vert_{\dot H^{s-1}}\ .\eneq

Now if we apply interpolation method, the estimates \eqref{highorderest} and
\eqref{highorderenergyest} with order 0 are direct consequences of
\eqref{6.7}, \eqref{6.3} and \eqref{6.8}. Next we will prove these
three estimates with order up to two.

\begin{prop}[Generalized Morawetz estimates]\label{lem2} Let $n\geq 3$, $s\in [0, 1)$ and $\rho>2$. Then for
the solution $u$ of the equation \eqref{LW} with $F=0$, we have
\beeq \label{lem2est}\sum_{|\al|\leq 2}\|Z^\al u\|_{L^2_t
Y_{s,\ep}}\lesssim
 \sum_{|\al|\leq 2} \left(\Vert Z^\al u_0\Vert_{{\dot H}^{s}}+
 \Vert Z^\al u_1\Vert_{\dot H^{s-1}}\right)\eneq
Moreover, if we assume only $\rho>1$ and $s\in [0,1]$, the estimate still holds with
$|\al|\le 1$.
\end{prop}

\begin{proof}
We first prove the estimate for $Z^\al=\pa_x$. Recall that for all
$-3/2\le\widetilde{\mu} < \mu \leq 3/2$, we have (Lemma 4.1 of
\cite{BoHa}) \beeq\label{a}\big\Vert \<x\>^{-\mu}
\widetilde{\partial} u \big\Vert_{L^2_x} \lesssim \big\Vert \<x\>^{-
\widetilde{\mu}} {P^{1/2}} u \big\Vert_{L^2_x}\ .\eneq
Also recall that $\tilde\pa=\pa g^{-1}$, a direct calculation induces
$$\sum_{|\al|\le 1} \|\pa_x^{\al} u\|_{L^2_t  Y_{s,\ep}}\les
\sum_{|\al|\le 1}
\|\tilde{\pa}_x^{\al} u\|_{L^2_t  Y_{s,\ep}},$$
Then for any
$\ep>0$, by \eqref{6.7},
\begin{eqnarray*}
\sum_{|\al|\le 1} \|\pa_x^{\al} u\|_{L^2_t  Y_{s,\ep}}&\les&
\sum_{|\al|\le 1}
\|\tilde{\pa}_x^{\al} u\|_{L^2_t  Y_{s,\ep}}\\
&\les& \sum_{j\le 1}\|P^{j/2} u\|_{L^2_t  Y_{s,\ep/2}}\\
&\les& \sum_{j\le 1}\left(\Vert P^{j/2} u_0\Vert_{\dot{H}^{s}}+\Vert
P^{j/2} u_1\Vert_{\dot H^{s-1}}\right)\\
&\les& \sum_{|\al|\le 1}\Vert \tilde\pa^\al
u_0\Vert_{\dot{H}^{s}}+\Vert u_1\Vert_{\dot H^{s-1}}+\Vert
P^{(1+s-1)/2} u_1\Vert_{L^2_x}
\\
&\les& \sum_{|\al|\le 1}\left(\Vert \pa_x^\al
u_0\Vert_{\dot{H}^{s}}+\Vert \pa_x^\al u_1\Vert_{\dot
H^{s-1}}\right),
\end{eqnarray*}
where we have used the inequalities \eqref{d}, \eqref{e} and Lemma \ref{leibniz} in the
last two inequalities (note $s\in [0,1]$).

Next we check with $Z^\al=\Om$. Recall that by the interpolation of \eqref{est1} and the duality of \eqref{est1}, we have
\beeq \label{duality}\|u\|_{L^2_tY_{s,\ep}}\leq
\|F\|_{L^2_tY_{1-s,\ep}'},\eneq
 if $u$ is a solution of
\eqref{LW} with vanishing initial data. Since $[P,\Omega]u=\sum_{|\al|\leq
2}r_{2-|\al|}\pa_x^\al u$, by using a combination of (6.7) in
\cite{SW} and
Lemma \ref{lem1} for $\Omega u$, we have
\begin{eqnarray}
 \|\Omega u\|_{L^2_t
Y_{s,\ep}}&\lesssim& \|\Omega u_0\|_{\dot{H}^s}+ \|\Omega
u_1\|_{\dot{H}^{s-1}}\nonumber\\
&&+\sum_{|\al|\leq 1}\|\<x\>^{3/2-s+\ep}r_{2-|\al|}\pa_x^\al
u\|_{L^2_{t,x}}+\|r_0\<x\>^{1/2+\ep}\pa_x^2u\|_{L_t^2\dot
H^{s-1}} \label{lem2est1}
\end{eqnarray}
 Now since $\rho>1$, by \eqref{a} and Lemma \ref{lem1},
\begin{eqnarray}
 \sum_{|\al|\leq 1}\|\<x\>^{3/2-s+\ep}r_{2-|\al|}\pa_x^\al u\|_{L^2_{t,x}}
 &\lesssim&
 \sum_{|\al|\leq 1}\|\<x\>^{-1/2-s-\ep'} \pa_x^\al u\|_{L^2_{t,x}}\nonumber\\
&\lesssim&
 \sum_{|\al|\leq 1}\|\<x\>^{-1/2-s-\ep'} \tilde\pa_x^\al u\|_{L^2_{t,x}}\nonumber\\
&\lesssim&
 \sum_{i\leq 1}\|\<x\>^{-1/2-s-{\ep'}/2} P^{i/2} u\|_{L^2_{t,x}}\nonumber\\
& \lesssim& \sum_{i\leq 1}\left(\|P^{i/2} u_0\|_{\dot
H^s}+\|P^{i/2} u_1\|_{\dot H^{s-1}}\right)\nonumber\\
& \lesssim& \sum_{|\al|\leq 1}\|\pa_x^\al u_0\|_{\dot
H^s}+\|\pa_x^\al u_1\|_{\dot H^{s-1}} \label{lem2est2}
\end{eqnarray}
where in the last inequality we have used the inequalities
\eqref{d}, \eqref{e} and Lemma \ref{leibniz}.

 Let $f(x)=r_0\<x\>^{1/2+\ep}=O(\<x\>^{-\rho+1/2+\ep})$. Then
 $f'(x)=O(\<x\>^{-\rho-1/2+\ep})$. Since $n\geq 3$,
 by Hardy's inequality with duality, the KSS estimates
\eqref{KSS} with Remark \ref{rem}, and interpolation,
 \begin{eqnarray}
 \|f\pa_x^2 u\|_{L_t^2\dot H^{s-1}} &\leq& \|\pa_x(f\pa_x u)\|_{L_t^2\dot H^{s-1}}+\|f'\pa_xu\|_{L_t^2\dot H^{s-1}}\nonumber\\
 &\lesssim& \|f\pa_xu\|_{L_t^2\dot H^{s}}+\|\left<x\right>f'\pa_xu\|_{L_t^2\dot H^{s}}\nonumber\\
 &\lesssim& \sum_{|\al|\leq 1}\|\pa_x^\al u_0\|_{\dot H^s}+\sum_{|\al|\leq 1}\|\pa_x^\al u_1\|_{\dot H^{s-1}}\ . \label{lem2est3}
 \end{eqnarray}
 On the basis of \eqref{lem2est1}, \eqref{lem2est2} and \eqref{lem2est3}, we are done with $Z^\al=\Om$.
 This completes the proof of the first order estimates under the
 condition $\rho>1$.

For the second order part, we first consider the case $Z^\al=\pa_x^2$. Since  $s\in [0,1)$, we can always find $\ep>0$ such that $1/2+s+\ep\le  3/2$. By Lemma \ref{lem7}, the proof for $Z^\al=\pa_x$, Lemma \ref{lem0} and Lemma \ref{lem6}, we have
\begin{eqnarray*}
\|\pa_x^2 u\|_{L^2_t  Y_{s,\ep}}
&\les&
\sum_{|\al|\leq 2}\|\tilde\pa_x^\al u\|_{L^2_t  Y_{s, \ep}}\\
&\les& \sum_{|\al|\le 1}\|\tilde\pa^\al u\|_{L^2_t  Y_{s,\ep}}+\|Pu\|_{L^2_t  Y_{s,\ep}}\\
&\les& \sum_{|\al|\le 1}\left(\Vert \pa_x^\al
u_0\Vert_{\dot{H}^{s}}+\Vert \pa_x^\al u_1\Vert_{\dot
H^{s-1}}\right)+\|Pu_0\|_{\dot H^s}+\|Pu_1\|_{\dot H^{s-1}}\\
&\approx& \sum_{|\al|\le 1}\left(\Vert \pa_x^\al
u_0\Vert_{\dot{H}^{s}}+\Vert \pa_x^\al u_1\Vert_{\dot
H^{s-1}}\right)+\|Pu_0\|_{\dot H^s}+\|P^{1/2}u_1\|_{\dot H^s}\\
&\les& \sum_{|\al|\le 1}\left(\Vert \pa_x^\al
u_0\Vert_{\dot{H}^{s}}+\Vert \pa_x^\al u_1\Vert_{\dot
H^{s-1}}\right)+\sum_{|\al|\leq 2}\|\pa_x^\al u_0\|_{\dot H^s}
+\sum_{|\al|\leq 1}\|\tilde\pa^\al u_1\|_{\dot H^{s}}\\
&\les& \sum_{|\al|\leq 2}\left(\|\pa_x^\al u_0\|_{\dot
H^s}+\|\pa_x^\al u_1\|_{\dot H^{s-1}}\right)\ ,
\end{eqnarray*}
where the fractional Lebniz rule (Lemma \ref{leibniz}) is used in the last inequality. Next, we consider the case $Z^\al=\Om^2$. Since
$[P,\Om^2]u=\sum_{|\al|\leq 3}\left( r_{2-|\al|}\pa_x^\al
u\right)$, and $\Om^2 u$ solves the wave equation with initial
data $(\Om^2 u_0, \Om^2 u_1)$ and nonlinear term $[P,\Om^2]u$, by
\eqref{duality}, Lemma \ref{lem1}, Lemma \ref{lem0.2} and the
higher order estimates we have proved,
\begin{eqnarray*}
\|\Om^2 u\|_{L^2_t Y_{{s,\ep}}}&\les& \|\Om^2u_0\|_{\dot H^s}+\|\Om^2 u_1\|_{\dot H^{s-1}}\\
&&+\sum_{|\al|\leq 2}\|\<x\>^{3/2-s+\ep}r_{2-|\al|}\pa_x^\al u\|_{L^2_{t,x}}+\sum_{|\al|= 3}\|\<x\>^{1/2+\ep}r_{2-|\al|}\pa_x^\al u\|_{L^2_{t}\dot H^{s-1}}\\
&\les& \|\Om^2u_0\|_{\dot H^s}+\|\Om^2 u_1\|_{\dot H^{s-1}}\\
&&+\sum_{|\al|\leq 2}\|\pa_x^\al u\|_{L^2_t Y_{s,\ep}}+\sum_{|\al|= 3}\|\<x\>^{1/2+\ep}r_{-1}\pa_x^\al u\|_{L^2_{t}\dot H^{s-1}}\\
&\les&\sum_{|\al|\leq 2}\left(\|Z^\al u_0\|_{\dot H^s}+\|Z^\al u_1\|_{\dot H^{s-1}}\right)+
\sum_{|\al|= 3}\|\<x\>^{1/2+\ep}r_{-1}\pa_x^\al u\|_{L^2_{t}\dot H^{s-1}}\\
&\les&\sum_{|\al|\leq 2}\left (\|Z^\al u_0\|_{\dot H^s}+\|Z^\al u_1\|_{\dot H^{s-1}}\right )
+\|\<x\>^{1+2\ep}r_{-1}\|_{L^\infty\cap\dot W^{1,n}}\|\<x\>^{-1/2-\ep}\pa_x^3 u\|_{L^2_{t}\dot H^{s-1}}\\
&\les&\sum_{|\al|\leq 2}\left (\|Z^\al u_0\|_{\dot H^s}+\|Z^\al u_1\|_{\dot H^{s-1}}\right )+
\sum_{|\al|\leq 2}\left (\|\pa_x^\al u_0\|_{\dot H^s}+\|\pa_x^\al u_1\|_{\dot H^{s-1}}\right )\\
&\les&\sum_{|\al|\leq 2}\left (\|Z^\al u_0\|_{\dot H^s}+\|Z^\al u_1\|_{\dot H^{s-1}}\right )
\end{eqnarray*}
where we have used the fact that $\rho>2$.

Since the commutator term $[P,\pa\Om]u=[P,\Om\pa]u=\sum_{|\al|\leq
3}\left (r_{3-|\al|}\pa_x^\al u\right)$ corresponds to an even
better case than what for $\Om^2$, the proof proceeds in the same
way. This completes the proof of the higher order estimates under
the conditions $\rho>2$ and $s\in [0,1)$.
\end{proof}

\begin{prop}[Higher order energy estimates]
\label{lem3} Let $n\geq 3$, $s\in [0, 1]$ and $\rho>2$. Then for
the solution $u$ of the equation \eqref{LW} with $F=0$, we have
 \beeq\label{lem3est} \sum_{|\al|\le 2}\|Z^\al u(t,x)
\|_{L^\infty_t
 \dot H^s}\les
 \sum_{|\al|\le 2}\left(\|
Z^\al  u_0\|_{\dot H^{s}}+\| Z^\al
 u_1\|_{\dot H^{s-1}}\right).\eneq
Moreover, if we assume only $\rho>1$, the estimate still holds with $|\al|\le 1$.
\end{prop}

\begin{proof}
By Lemma \ref{lem0} and elliptic regularity for $P$,
we know
$$\|\pa_x u\|_{\dot H^1}\les \|\pa^2_x u\|_{L^2}\les \|Pu\|_{L^2}+\|u\|_{L^2}\les \|P^{1/2} u\|_{\dot H^1}+\|P^{1/2}u\|_{\dot H^{-1}}.$$
Interpolating this estimate with \eqref{d} with $s=1$, $\|\pa_x
u\|_{L^2}\simeq \|P^{1/2}u\|_{L^2}$, we get that for $s\in [0,1]$,
\begin{eqnarray}
\|\pa_x u\|_{\dot H^s}&\les& \|P^{1/2} u\|_{\dot H^s}+\|P^{1/2}u\|_{\dot H^{-s}}\\
&\les& \|P^{1/2} u\|_{\dot H^s}+\|u\|_{\dot H^{1-s}}\nonumber.
\end{eqnarray}
Thus by Lemma \ref{lem0} we have for $s\in [0, 1/2]$ (such that
$s\le 1-s$ and $\dot H^s\cap \dot H^{1+s}\subset \dot H^{1-s}$),
\begin{eqnarray}
\sum_{|\al|\le 1}\|{\pa}^\al_x u\|_{L_t^\infty \dot H^s} &\les&
\sum_{j\le 1}\|P^{j/2} u\|_{L_t^\infty \dot H^s}+\|u\|_{L_t^\infty\dot H^{1-s}}\nonumber\\
&\les &\sum_{|\al|\le 1}\left(\| {\pa}_x^\al u_0\|_{\dot H^{s}}+\|
\pa_x^\al u_1\|_{\dot H^{s-1}}\right)+\|
u_0\|_{\dot H^{1-s}}+\| u_1\|_{\dot H^{-s}}\nonumber\\
&\les &\sum_{|\al|\le 1}\left(\| {\pa}_x^\al u_0\|_{\dot H^{s}}+\|
\pa_x^\al u_1\|_{\dot H^{s-1}}\right)\ .\label{lem3est2}
\end{eqnarray}
Now we can deal with $\Om u$. Noticing that
$$
\tilde{\Om}_{ij}f=g^{-1}\Om_{ij} f+(x_i\pa_jg^{-1}-x_j\pa_ig^{-1})f\
,$$ by the fractional Leibniz rule, we have
$$\|\tilde{\Om} f\|_{\dot
H^s}\les\sum_{|\al|\leq 1}\|\Om^\al f\|_{\dot H^s}\ ,\ |s|<n/2\ .$$
We have similar relationship between $\pa_x u$ and $\tilde\pa_x u$.
By the Sobolev embedding, for any $h\in L^n$, we have
\begin{eqnarray}
\|\<x\>^{-1/2-\ep} h u\|_{\dot H^{s-1}}&\les&
\|\<x\>^{-1/2-\ep}h u\|_{L^{2n/(n+2(1-s))}}\nonumber\\
&\les&\|h\|_{L^n}\|\<x\>^{-1/2-\ep}u\|_{L^{{2n}/({n-2s})}}\nonumber\\
&\les&\|\<x\>^{-1/2-\ep}u\|_{\dot H^s}. \label{useH1}
\end{eqnarray}
Thus by the energy estimate, Lemma \ref{lem0.1}, \ref{leibniz}
and \ref{lem0.2}:
\begin{eqnarray*}
\|\tilde{\Om} u\|_{L_t^\infty\dot H^s}
&\les&
\|\tilde{\Om} u_0\|_{\dot H^s}+\|\tilde{\Om} u_1\|_{\dot
H^{s-1}}+\|\<x\>^{1/2+\ep}[P,\tilde{\Om}]u\|_{L^2_t\dot H^{s-1}}\nonumber\\
&\les& \|\tilde{\Om} u_0\|_{\dot H^s}+\|\tilde{\Om} u_1\|_{\dot
H^{s-1}}+\sum_{1\leq|\al|\leq
2}\|r_{2-|\al|}\<x\>^{1/2+\ep}\tilde{\pa}_x^\al u\|_{L^2_t\dot
H^{s-1}}\\
&\les&\sum_{|\al|\leq 1}\left (\|\Om^\al u_0\|_{\dot H^s}+\|\Om^\al
u_1\|_{\dot H^{s-1}}\right )\\
&&+\sum_{1\leq|\al|\leq
2}\|r_{2-|\al|}\<x\>^{1+2\ep}\|_{L^\infty\cap \dot
H^{1-s,n/(1-s)}}\|\<x\>^{-1/2-\ep}\tilde{\pa}_x^\al
u\|_{L^2_t\dot H^{s-1}}\nonumber\\
&\les&\sum_{|\al|\leq 1}\left(\|\Om^\al u_0\|_{\dot H^s}+\|\Om^\al
u_1\|_{\dot H^{s-1}}\right)+\sum_{1\leq|\al|\leq
2}\|\<x\>^{-1/2-\ep}{\pa}_x^\al
u\|_{L^2_t\dot H^{s-1}}\nonumber\\
&&+\|\<x\>^{-1/2-\ep}(\pa g^{-1})u\|_{L^2_t \dot H^{s-1}}+
\|\<x\>^{-1/2-\ep}[\pa(g^{-1}\pa g^{-1})] u\|_{L^2_t \dot H^{s-1}}
\nonumber
\\
&\les&\sum_{|\al|\leq 1}\left (\|Z^\al u_0\|_{\dot
H^s}+\|Z^\al u_1\|_{\dot H^{s-1}}\right )\ ,
\end{eqnarray*}
where we have used the fact that $\rho>1$ and \eqref{useH1} with $h=\pa g^{-1}$, $\pa g^{-2}$ and $h=\pa(g^{-1}\pa g^{-1})$ (the condition $h\in L^n$ is satisfied since the condition \eqref{H1} on the metric $g$). Noticing that $\Om u=g\tilde\Om u-g(\Om g^{-1})u$, we hence have
\beeq\label{lem3est3}\|{\Om} u\|_{L_t^\infty\dot H^s}\les
\sum_{|\al|\leq1}\|\tilde{\Om}^\al u\|_{L_t^\infty\dot H^s} \les
\sum_{|\al|\leq 1}\left(\|Z^\al u_0\|_{\dot H^s}+\|Z^\al
u_1\|_{\dot H^{s-1}}\right)\ .\eneq On the basis of \eqref{lem3est2} and
\eqref{lem3est3}, we complete the proof of the energy estimates of
order one, under the conditions $s\in [0, 1/2]$ and $\rho>1$.

For the part with second order derivatives, we need only to deal with
$\pa_x^2$ and $\Om^2$ as before.

By Lemma \ref{lem6} and Lemma \ref{lem0}, we have
\begin{eqnarray}
\|{\pa}^2_x u\|_{L_t^\infty \dot H^s} &\les&
\|Pu\|_{L^\infty_t\dot H^s}+ \| u\|_{L_t^\infty \dot H^s}\nonumber\\
&\les &\| Pu_0\|_{\dot H^{s}}+\| P u_1\|_{\dot H^{s-1}}+\|
u_0\|_{\dot H^{s}}+\|  u_1\|_{\dot
H^{s-1}} \nonumber\\
&\les &\sum_{|\al|\le 2}\left(\| \pa_x^\al u_0\|_{\dot H^{s}}+\|
\pa_x^\al u_1\|_{\dot H^{s-1}}\right)\ .\label{lem3est4}
\end{eqnarray}
Here we remark that we can control $\sum_{|\al|= 1}\|{\pa}^\al_x
u\|_{L_t^\infty \dot H^s}$ for $s\in [0, 1]$ instead of the
restriction $s \in [0,1/2]$ in \eqref{lem3est2}, by \eqref{lem3est4}
and \eqref{6.7}, which enables us to relax the condition to $s\in [0,1]$ in the estimates of order one.

By Lemma \ref{lem0.1}, Lemma \ref{lem0.2}, 
and what we have gained in previous steps, 
if $\rho>2$,
\begin{eqnarray}
\|\Om^2 u\|_{L_t^\infty\dot H^s}
&\les&
\sum_{|\al|\le 2}\|\tilde{\Om}^\al u\|_{L_t^\infty \dot H^s}\nonumber\\
&\les& \sum_{|\al|\le 2}\left (\|\tilde{\Om}^2 u_0\|_{\dot H^s}+\|\tilde{\Om}^2 u_1\|_{\dot
H^{s-1}}\right )+\sum_{1\leq |\al|\leq
3}\|r_{2-|\al|}\<x\>^{1/2+\ep}\tilde \pa_x^\al u\|_{L^2_t\dot
H^{s-1}}\nonumber\\
&\les& \sum_{|\al|\le 2}\left (\|\tilde{\Om}^2 u_0\|_{\dot
H^s}+\|\tilde{\Om}^2 u_1\|_{\dot H^{s-1}}\right )+\sum_{1\leq
|\al|\leq 3}\|\<x\>^{-1/2-\ep} \pa_x^\al u\|_{L^2_t\dot
H^{s-1}}\nonumber\\
&\les&\sum_{|\al|\leq 2}\left(\|Z^\al u_0\|_{\dot H^s}+\|Z^\al
u_1\|_{\dot H^{s-1}}\right)\ .\label{lem3est5}
\end{eqnarray}
We are done with the second order estimates based on
\eqref{lem3est4} and \eqref{lem3est5}.
\end{proof}

\begin{prop}[Sobolev inequality with angular smoothing]
\label{lem4} Let $u$ be a solution of \eqref{LW} with $F=0$ and
$n\geq 3$. Then for any $s\in(1/2,1]$ and $\rho>1$, there exists a suitable $\eta>0$ so that
we have: \beeq\sum_{|\al|\le 1}\||x|^{{n}/{2}-s} Z^\al  u(t,x)
\|_{L^\infty_{t, |x|} L^{2+\eta}_\omega}\les \sum_{|\al|\le
1}\left(\| Z^\al u_0\|_{\dot H^{s}}+\| Z^\al u_1\|_{\dot
H^{s-1}}\right)\eneq Furthermore, if we assume $\rho>2$, then we
have \beeq\sum_{|\al|\le 2}\||x|^{{n}/{2}-s} Z^\al u(t,x)
\|_{L^\infty_{t, |x|} L^{2+\eta}_\omega}\les \sum_{|\al|\le
2}\left(\| Z^\al u_0\|_{\dot H^{s}}+\| Z^\al u_1\|_{\dot
H^{s-1}}\right)\eneq
\end{prop}

\begin{proof}
This is a direct consequence of the energy estimates Proposition
\ref{lem3} and the inequality \eqref{trace-g}.
\end{proof}

\begin{prop}[Local energy estimates]
\label{lem5} Assume $n\geq 3$, let $s\in [0,1]$, $p\ge 2$,
$k=0,1,2$, $\rho>k$ and $u$ be a solution of \eqref{LW} with $F=0$.
We have \beeq\label{lem5est}\sum_{|\al|\le k}\|\phi Z^\al
u\|_{L^p_t \dot{H}^s} \les \sum_{|\al|\le k }\left(\Vert Z^\al
u_0\Vert_{\dot{H}^{s}}+\Vert Z^\al u_1\Vert_{\dot H^{s-1}}\right)\
,\eneq where $\phi\in C_0^\infty(\R^n)$.
\end{prop}

\begin{proof} The estimate with $k=0$ is just \eqref{6.8}.
For the higher order estimates with $|\al|=k\ge 1$, by the higher
order KSS estimates \eqref{KSS},
\begin{eqnarray*}\|\phi  Z^\al
u\|_{L^2_t \dot H^1}&\les&\|\phi \ \pa_x Z^\al u\|_{L^2_{t,x}}+\|\phi'\  Z^\al u\|_{L^2_{t,x}}\\
&\les& \|\<x\>^{-1/2-\ep}\pa_x Z^\al u\|
_{L^2_{t,x}}+\|\<x\>^{-3/2-\ep}Z^\al u\|_{L^2_{t,x}}\\&\les&
\sum_{|\al|\le k} \left( \|Z^\al u_0\|_{\dot H^1}+ \|Z^\al
u_1\|_{L^2}\right).
\end{eqnarray*}
For $s=0$, note that $\phi\ \Omega=r_0 \pa_x$,
\begin{eqnarray*}
\|\phi Z^\al u\|_{L^2_{t,x}}&\les& \|\<x\>^{-1/2-\ep} \pa_x
Z^{\al-1}u\|_{L^2_{t,x}}\\&\les& \sum_{|\al|\le k-1}\left(\|Z^\al
u_0\|_{\dot H^1}+ \| Z^\al u_1\|_{L^2}\right)\\
&\les& \sum_{|\al|\le k}\left(\|Z^\al u_0\|_{L^2}+ \| Z^\al
u_1\|_{\dot H^{-1}}\right)\ .
\end{eqnarray*}
By interpolation between the above two estimates, we get \eqref{lem5est} with $p=2$.
This will complete the proof if we combine it with the energy estimates
in Proposition~\ref{lem3}.
\end{proof}

\noindent\textbf{Proof of Theorem \ref{mainest}}: From the above
four propositions, we have proved the higher order version of
\eqref{6.7}, \eqref{6.3} and \eqref{6.8}, which gives us the
required higher order estimates \eqref{highorderest} and
\eqref{highorderenergyest}.\qed

\section{Local in Time Strichartz Estimates}
In this section, we give the proof of Thoerem \ref{mainest2}. The
first lemma is concerned with the KSS estimates for the perturbed wave equation,
obtained in  Theorem 2.1 of \cite{HiWaYo} (see also Theorem 5.1 in
\cite{MeSo06_01}).
\begin{lem}
\label{KSSlocal} Let $n\ge 3$, $\Box_h=\pt^2-\Delta+h^{\al\be}(t,x)\pa_\al
\pa_\be$, $h^{\al\be}=h^{\be\al}$ and $\sum |h^{\al\be}|\le 1/2$.
Then the solution to the equation $\Box_h u=F$ satisfies
\begin{eqnarray}
&&(1+T)^{-2a}\big\Vert |x|^{-1/2+a} (|u'|+\frac{|u|}{|x|})
\big\Vert^2_{L^2 ( [0, T] \times \R^{n} )}\nonumber \\
&+&
\big\Vert \< x \>^{-1/2-\ep} (|u'|+\frac{|u|}{|x|})
\big\Vert_{L^2 ( [0, T] \times \R^{n} )}^2
\nonumber \\
 &\lesssim&    \Vert u' ( 0 , \cdot ) \Vert^2_{L^2 ( \R^{n} )}+\int_0^T \int (u'+\frac u{|x|})(|F|+(|h'|+\frac h{|x|})|u'|) \,dxdt \label{70-est-KSS-SmallPert}
\end{eqnarray} for any $\ep>0$ and $a\in (0, 1/2)$.\\
\end{lem}

On the basis of the KSS estimates for wave equations with variable
coefficients and local energy decay \eqref{lem5est}, we can adapt
the arguments in \cite{SW} to obtain the following KSS estimates
for asymptotically Euclidean manifolds.

\begin{prop}
\label{Prop3.2}
Assume that \eqref{H1} and \eqref{H2} hold with $\rho > 1$.  Let $N
\ge 0$, $ 0<\mu<1/2$. Then the solution of \eqref{LW} satisfies
\begin{multline}\label{KSS2}
\sum_{\vert \alpha \vert \leq N} (1+T)^{\mu-1/2}
 \big\Vert \<x\>^{-\mu} \left(|(\Gamma^{\alpha} u)'|+  \frac{|\Gamma^{\alpha} u|}{\<x\>}\right)
 \big\Vert_{L^2_T L^2_x} \\
 \lesssim \sum_{\vert \alpha \vert \leq N} \big\Vert (Z^{\alpha}u)'
(0, \cdot ) \big\Vert_{L^2_x}
 + \sum_{\vert \alpha \vert \leq N}\big\Vert
 \Gamma^{\alpha} F(s, \cdot ) \big\Vert_{L^1_T
 L^2_x}\ ,
\end{multline} where $L^q_T L^r_x=L^q([0,T]; L^r(\R^n))$.
\end{prop}

As a consequence of this KSS estimate, similarly to the previous
proof of Proposition \ref{lem2}, we can have the following
estimates.
\begin{coro}\label{coro} Assume that \eqref{H1} and \eqref{H2} hold with $\rho > 2$.
Let $ 0<\mu\le 1/2$ and
\begin{equation*}
A_{\mu} (T) = \left\{ \begin{aligned}
&(\log (2+T))^{-1/2} &&\mu = 1/2 ,   \\
&(1+T)^{\mu- 1/2} &&0<\mu < 1/2 .
\end{aligned} \right.
\end{equation*} We have \beeq
\label{KSS3}\|\<x\>^{-\mu}e^{itP^{1/2}}f\|_{L^2_T L^2_x}\les
 A_\mu(T)^{-1}\|f\|_{L^2}.\eneq
 Moreover, if $0<\mu<1/2$, for the solution $u$ of the
equation \eqref{LW} with $F=0$, we have \beeq
\label{KSS4}\sum_{|\al|\le 2}\|\<x\>^{-\mu} Z^\al u\|_{L^2_T
L^2_x}\les T^{1/2-\mu+\ep} \sum_{|\al|\leq 2} \left(\Vert Z^\al
u_0\Vert_{L^2}+
 \Vert Z^\al u_1\Vert_{\dot H^{-1}}\right).\eneq
And, if we assume $\rho>1$ instead of $\rho>2$, we have
 the same estimates of first order $(|\al|\le 1)$.
\end{coro}

\begin{proof}
\eqref{KSS3} is a direct consequence if we employ \eqref{KSS2} with $\al=0$ for $u'=\pa_t u$. To obtain \eqref{KSS4}, we basically follow the argument as in Proposition \ref{lem2} with some modifications.
For the second order part, we first consider the case $Z^\al=\pa_x^2$.
We claim that we have the following inequality
\beeq
\label{cor3.3fact}\|\<x\>^{-\mu}\pa_x u\|_{L^2_x}\le \ep \|\<x\>^{-\mu}\pa_x^2 u\|_{L^2_x}+C(\ep) \|\<x\>^{-\mu} u\|_{L^2_x}\ .
\eneq
  By Lemma \ref{lem7}, Lemma \ref{lem0} and Lemma \ref{lem6}, we have
\begin{eqnarray*}
A_\mu(T)\|\<x\>^{-\mu}\pa_x^2 u\|_{L^2_T  L^2_x}
&\les&
A_\mu(T)\sum_{|\al|\leq 2}\|\<x\>^{-\mu} \tilde\pa_x^\al u\|_{L^2_T L^2_x}\\
&\les& A_\mu(T)\sum_{|\al|\leq 1}\|\<x\>^{-\mu} \tilde\pa_x^\al u\|_{L^2_T L^2_x}+A_\mu(T)\|\<x\>^{-\mu} P u\|_{L^2_T L^2_x}\\
&\les& A_\mu(T)\sum_{|\al|\leq 1}\|\<x\>^{-\mu} \pa_x^\al u\|_{L^2_T L^2_x}+A_\mu(T)\|\<x\>^{-\mu} P u\|_{L^2_T L^2_x}\\
&\les& \ep A_\mu(T)\|\<x\>^{-\mu}\pa_x^2 u\|_{L^2_T  L^2_x}
+C(\ep)A_\mu(T) \|\<x\>^{-\mu}  u\|_{L^2_T L^2_x}\\
&&+A_\mu(T)\|\<x\>^{-\mu} P u\|_{L^2_T L^2_x}\\
&\les& \ep A_\mu(T)\|\<x\>^{-\mu}\pa_x^2 u\|_{L^2_T  L^2_x}
\\
&&+C(\ep)\left(\Vert
u_0\Vert_{L^{2}}+\Vert u_1\Vert_{\dot H^{-1}}\right)+\|Pu_0\|_{L^2}+\|Pu_1\|_{\dot H^{-1}},\end{eqnarray*}
where we have used \eqref{KSS3} and \eqref{cor3.3fact}.
Hence we have
\begin{eqnarray*}
A_\mu(T)\|\<x\>^{-\mu}\pa_x^2 u\|_{L^2_T  L^2_x}
&\les&\Vert
u_0\Vert_{L^2}+\Vert u_1\Vert_{\dot
H^{-1}}+\|Pu_0\|_{L^2}+\|Pu_1\|_{\dot H^{-1}}\\
&\les& \Vert
u_0\Vert_{L^2}+\Vert u_1\Vert_{\dot
H^{-1}}+\|Pu_0\|_{L^2}+\|P^{{1}/2}u_1\|_{L^2}\\
&\les& \Vert  u_1\Vert_{\dot
H^{-1}}+\sum_{|\al|\leq 2}\|\pa_x^\al u_0\|_{L^2}
+\|\tilde\pa u_1\|_{L^2}\\
&\les& \sum_{|\al|\leq 2}\left(\|\pa_x^\al u_0\|_{L^2}+\|\pa_x^\al u_1\|_{\dot H^{-1}}\right)\ .
\end{eqnarray*}

Now we are left with the norm for $Z=\Om,\Om^2$, but from the proof of Proposition \ref{lem2}, we know it suffices to prove the following estimates
\beeq
\label{localcommutator}
\|\<x\>^{-\mu}w\|_{L^2_{t,x}([0,T]\times \R^n)}\les T^{1/2-\mu+\ep}\|\<x\>^{1/2+\ep}F\|_{L^2_t\dot H^{-1}([0,T]\times\R^n)},
\eneq
if $w$ is the solution of \eqref{LW} with vanishing initial data.
Recall that we have proved in Lemma \ref{lem1} that
\beeq
\label{1}
\|\<x\>^{- 1/2-\ep}w\|_{L^2_{t,x}([0,T]\times \R^n)}\les \|\<x\>^{1/2+\ep}F\|_{L^2_{t}\dot H^{-1}([0,T]\times\R^n)}.
\eneq
Also if we restrict the time $t$ in $[0,T]$, it is easy to verify that Lemma \ref{lem0.1} still holds, i.e.
\beeq
\label{2}
\|w\|_{L^{2}_tL^2_{x}([0,T]\times \R^n)}\les T^{1/2}\|w\|_{L^{\infty}_tL^2_{x}([0,T]\times \R^n)}\les T^{1/2}\|\<x\>^{1/2+\ep}F\|_{L^2_{t}\dot H^{-1}([0,T]\times\R^n)}.
\eneq
Now \eqref{localcommutator} just follows from the interpolation between \eqref{1} and \eqref{2}. To conclude the proof of \eqref{KSS4}, it remains to prove the claim \eqref{cor3.3fact}.

{\noindent\bf Proof of \eqref{cor3.3fact}.} This inequality is true for $\mu=0$. For general $\mu\ge 0$, we apply the estimate for $\mu=0$ to $v=\phi u$ with $\phi=\psi(x/R), \psi\in C^\infty, 0\leq\psi\leq 1, \text{supp} \psi\subset\{1/4<|x|<2\}$, $\psi=1$ in $B_1\backslash B_{1/2}$ and $R\ge 1$. Because of $\{x:\phi(x)=1\}\subset\{|x|>R/4\}$ and supp$\phi\subset\{R/4<|x|<2R\}$, we get
\begin{eqnarray*}
\|\<x\>^{-\mu}\pa_x  u\|_{L^2({\{x:\phi(x)=1\}})}&=& \|\<x\>^{-\mu}\pa_x (\phi u)\|_{L^2(\{x:\phi(x)=1\})}\nonumber\\
 &\le& C R^{-\mu}\|\pa_x (\phi u)\|_{L^2(\R^n)}\nonumber\\
 &\le& C R^{-\mu}(\ep \|\pa_x^2 (\phi u)\|_{L^2(\R^n)}+C(\ep) \| \phi u\|_{L^2(\R^n)})\nonumber\\
&\le& C(\ep \|\<x\>^{-\mu}\pa_x^2 (\phi u)\|_{L^2(\R^n)}+C(\ep) \|\<x\>^{-\mu} \phi u\|_{L^2(\R^n)})\nonumber\\
&\le& C\left(\ep \|\<x\>^{-\mu}\pa_x^2 u\|_{L^2(\text{supp} \phi)}+C\ep R^{-1}\|\<x\>^{-\mu}\pa_x u\|_{L^2(\text{supp} \phi')}+\right.\nonumber\\
&&\left.(C(\ep)+C \ep R^{-2}) \|\<x\>^{-\mu}  u\|_{L^2(\text{supp} \phi)}\right)\ .
\end{eqnarray*}
If we choose instead $\psi=1$ in $B_1$ and $0$ for $|x|\ge 2$, then
\begin{eqnarray*}
  \|\<x\>^{-\mu}\pa_x  u\|_{L^2(\{x:|x|\le 1\})} &\le& C\ep \|\<x\>^{-\mu}\pa_x^2 u\|_{L^2(\{x:|x|\le 2\})}+\\
  &&C \ep \|\<x\>^{-\mu}\pa_x u\|_{L^2(\{x:|x|\le 2\})}+(C(\ep)+C\ep ) \|\<x\>^{-\mu}  u\|_{L^2(\{x:|x|\le 2\})}\ .\end{eqnarray*}

Combining the above two inequalities, we see
$$\|\<x\>^{-\mu}\pa_x  u\|_{L^2(\R^n)} \le C \ep \|\<x\>^{-\mu}\pa_x^2 u\|_{L^2(\R^n)}+C\ep \|\<x\>^{-\mu-1}\pa_x u\|_{L^2(\R^n)}+C (C(\ep)+\ep) \|\<x\>^{-\mu}  u\|_{L^2(\R^n)}\ ,$$
which implies \eqref{cor3.3fact}, by choosing small enough $\ep>0$.
\end{proof}

The next estimate is based on the endpoint trace lemma.
\begin{prop}
Let $\dot B_{pq}^s$ denote the homogeneous Besov space. Then we have
\beeq\label{trace} \||x|^{({n-1})/2}e^{itP^{1/2}}f\|_{L^\infty_t
L^\infty_r L^2_{\om}}\les \|f\|_{\dot B_{2,1}^{1/2}}\,. \eneq
\end{prop}
\begin{proof}
Recall that we have the endpoint Trace lemma (see (1.7) in
\cite{FaWa}): \beeq
\label{endpointtrace}r^{({n-1})/2}\|f (r\cdot)\|_{L^2_{\om}}\les
\|f\|_{\dot B_{2,1}^{1/2}}\ ,\eneq which gives that \beeq
\label{trace1}
\||x|^{({n-1})/2}e^{itP^{1/2}}f\|_{L^\infty_rL^2_{\om}}\les
\|e^{itP^{1/2}}f\|_{\dot B_{2,1}^{1/2}}\ .\eneq On the other hand,
by Lemma \ref{lem0} we have
$$\|e^{itP^{1/2}}f\|_{\dot H^{1}}\les \|P^{1/2}e^{itP^{1/2}}f\|_{L^2_x}\les \|P^{1/2}f\|_{L^2_x}\les \|f\|_{\dot H^1}.$$
Noticing that $\|f\|_{\dot B_{2,2}^{s}}=\|f\|_{\dot H^s}$, we can
rewrite the above estimate as
$$\|e^{itP^{1/2}}f\|_{\dot B_{2,2}^{1}}\les \|f\|_{\dot B_{2,2}^1}\ .$$
Interpolating this estimate with the energy estimate
$$\|e^{itP^{1/2}}f\|_{\dot B_{2,2}^{0}}\les \|f\|_{\dot B_{2,2}^0}$$
gives \beeq \label{trace2}\|e^{itP^{1/2}}f\|_{\dot
B_{2,1}^{1/2}}=\|e^{itP^{1/2}}f\|_{(\dot B_{2,2}^{1},\dot
B_{2,2}^0)_{ 1/2,1}}\les \|f\|_{(\dot B_{2,2}^1,\dot
B_{2,2}^0)_{ 1/2,1}}=\|f\|_{\dot B_{2,1}^{1/2}}\ , \eneq where
we have used the fact that (Theorem 6.4.5 in \cite{JJ})
$$(\dot B_{pq_0}^{s_0}, \dot B_{pq_1}^{s_1})_{\theta,r}= \dot B_{pr}^{s^*},\ \text{if}\ s_0\neq s_1,~ 0<\theta<1,~ r,q_0,q_1\geq 1 ~\text{and}~ s^*=(1-\theta)s_0+\theta s_1.$$
Now our estimate \eqref{trace} follows from \eqref{trace1} and
\eqref{trace2}.
\end{proof}

Now we are ready to obtain the local in time Strichartz estimates as
follows.
\begin{prop}
\label{prop3.4} 
Let $2\leq p <\infty$ and $a\in (0, 1/p)$. Then we have
\beeq\label{1storder-local-stri}\|\<x\>^{-a}|x|^{(n-1)(
1/2- 1/p)} e^{itP^{1/2}}f\|_{L^p_T L^p_r L^2_\om}\les (1+T)^{
1/p-a} \|f\|_{\dot H^{1/2-1/p}}.\eneq

\end{prop}
\begin{proof}
This estimate follows from the real interpolation between
\eqref{KSS3} and \eqref{trace} with $\theta= 2/p$ (for similar
arguments, see, e.g., \cite{H}, \cite{Yu09}).
\end{proof}

Finally we give the proof of Theorem \ref{mainest2}.

\noindent \textbf{Proof of Theorem \ref{mainest2}}: Since the estimates
in Theorem \ref{mainest2} with order $0$ are just obtained in Proposition
\ref{prop3.4}, we are left with the higher order estimates.
Similarly to the proof of Proposition \ref{prop3.4}, we need only to
show the higher order estimates that correspond to \eqref{KSS3}
and \eqref{trace}.

The higher order estimates corresponding to \eqref{KSS3} are known
from Corollary \ref{coro}. For the higher order estimates  of
\eqref{trace}, by \eqref{endpointtrace} we have
\beeq\label{trace-inter} \sum_{|\al|\leq
2}\||x|^{({n-1})/2}Z^\al u(t,\cdot)\|_{L^\infty_rL^2_{\om}}\les
\sum_{|\al|\leq 2}\|Z^\al u(t,\cdot)\|_{\dot B_{2,1}^{1/2}}~. \eneq
On the other hand, from the energy estimates in Proposition
\ref{lem3}, we have for any $s\in [0,1]$
$$\sum_{|\al|\leq 2}\|Z^\al
u(t,\cdot)\|_{\dot H^s} \les \sum_{|\al|\leq 2} \left(\Vert Z^\al
u_0\Vert_{\dot H^s}+
 \Vert Z^\al u_1\Vert_{\dot H^{s-1}}\right)\ .$$
Now the real interpolation between the above two estimates with
$s=0$ and $s=1$ gives
$$\sum_{|\al|\leq 2}\|Z^\al
u(t,\cdot)\|_{\dot B^{1/2}_{2,1}} \les \sum_{|\al|\leq 2}
\left(\Vert Z^\al u_0\Vert_{\dot B^{1/2}_{2,1}}+
 \Vert Z^\al u_1\Vert_{\dot B^{-1/2}_{2,1}}\right)\ .$$
Combining this estimate with \eqref{trace-inter}, we get the second order
estimates of \eqref{trace}, which completes the proof of Theorem
\ref{mainest2} for $\rho>2$. When $\rho>1$, we need only to use \eqref{trace-g} instead of \eqref{endpointtrace}. \qed

\section{Strauss Conjecture when $n=3,4$}

In this section, we will prove the existence results in Theorem
\ref{Strauss} and Theorem \ref{Strauss2}.

\subsection{Global results when $n=3,4$}

In this subsection, we prove the Strauss conjecture stated in
Theorem \ref{Strauss}. The result when $n=3$ and $\rho>1$ has been
proved in \cite{SW}, under the additional assumption that $g_{ij}$
is spherically symmetric. Since we have obtained the same
estimates without this assumption, the existence result with a
general metric follows from the same argument. Here we present the
proof for $n=3,4$ under the conditions $\rho>2$ and $p>p_c$,
and we are following the argument as in \cite{HMSSZ}.

We define $X=X_{s,\ep,q}(\R^n)$ to be the space with norm defined by
\begin{equation}\label{Xspace}
\|h\|_{X_{s,\ep,q}}=\|h\|_{L^{q_s}(|x|\le 1)}\, + \, \bigl\| \,
|x|^{{n}/{2}-({n+1})/{q}-s-\ep} h\bigr\|_{L^q_{|x|}
L^{2+\eta}_\omega(|x|\ge 1)} ,\end{equation} where  $
n\bigl(\tfrac12-\tfrac1{q_s})=s$.
Combining the Sobolev inequalities with angular regularity \eqref{trace-g} with Sobolev
embedding $\dot H^s\subset L^{q_s}$, we have the embedding
$$\dot H^{s}\subset X_{s,0,\infty}$$ for $s\in(1/2,n/2)$ and some
$\eta>0$. By duality, we have (see Theorem 2.11 of \cite{LZ})
\beeq\label{70-est-Embedding}X_{1-s,0,\infty}'\subset \dot H^{s-1}
\textrm{ for }s\in \left(({2-n})/2, 1/2\right).\eneq

With these notations, Theorem \ref{mainest} tells us that for the
solution $u$ to the linear wave equation $\pt^2 u+Pu=0$, we have
$$\sum_{|\al|\le 2}\left(
\|Z^\al u \|_{L^\infty_t \dot H^s\cap L^p_t X_{s,\ep,p}}+\|\pt Z^\al
u\|_{L^\infty_t \dot H^{s-1}}\right)\les  \sum_{|\al|\le 2}
\left(\Vert Z^\al u_0\Vert_{\dot{H}^{s}}+
 \Vert Z^\al  u_1\Vert_{\dot H^{s-1}}\right)$$ for $s\in (1/2-1/p, 1)$.
By Duhamel's formula and \eqref{70-est-Embedding}, we see that for
$u$ solving the linear wave equation $\pt^2 u+Pu=F$, we have
\begin{eqnarray}
\label{70-est-Strauss-KeyHigh2}&\sum_{|\al|\le 2}\left( \|Z^\al u
\|_{L^\infty_t \dot H^s\cap L^p_t X_{s,\ep,p}}+\|\pt Z^\al
u\|_{L^\infty_t \dot
H^{s-1}}\right) & \\
\les& \sum_{|\al|\le 2} \left(\Vert Z^\al u_0\Vert_{\dot{H}^{s}}+
 \Vert Z^\al  u_1\Vert_{\dot H^{s-1}}+
 \|Z^\al  F\|_{L^1_t \dot H^{s-1}}\right)&\nonumber\\
\les& \sum_{|\al|\le 2} \left(\Vert Z^\al u_0\Vert_{\dot{H}^{s}}+
 \Vert Z^\al  u_1\Vert_{\dot H^{s-1}}+
 \|Z^\al  F\|_{L^1_t X_{1-s, 0, \infty}'}\right)&\nonumber
 \end{eqnarray}  if $\rho>2$, $p>2$, $s\in (1/2-1/p,1/2)$.


For the linear wave equation $(\pt^2-\Delta_g)u=F$, using the
observation \eqref{70-est-EquiEqn}, we have the same set of estimates.

Let us now see how we can use these estimates to prove Theorem
\ref{Strauss}. Considering the Cauchy data $(u_0,u_1)$ satisfying
the smallness condition \eqref{70-eqn-SLW-data}, set $u^{-1}\equiv
0$ and let $u^{(0)}$ solve the Cauchy problem \eqref{eq} with
$F=0$. We iteratively define $u^{(k)}$, for $k\ge 1$, by solving
$$
\begin{cases}
(\partial^2_t-\Delta_{g})u^{(k)}(t,x)=F_p(u^{(k-1)}(t,x))\,,
\quad (t,x)\in \R_+\times \R^n,
\\
u(0,\cdot)=u_0, \quad
\partial_t u(0,\cdot)=u_1.
\end{cases}
$$
Let $s=s_c-{p\ep}/({p-1})= n/2- 2/({p-1})-p\ep/({p-1})$, our aim is to show that if the constant $\de>0$ in
\eqref{70-eqn-SLW-data} is small enough, then so is
$$
M_k = \sum_{|\alpha|\le 2}\left( \|Z^\alpha u^{(k)}
\|_{L^\infty_t\dot H^s \cap L^p_t X_{s,\ep,p}}+ \|\partial_t
Z^\alpha u^{(k)}\|_{L^\infty_t \dot H^{s-1}} \right)
$$for every $k=0,1,2,\dots$~. Notice that since $p_c<p<1+ 4/({n-1})$, we can always choose $\ep>0$ small enough so that $s\in ( 1/2- 1/p,  1/2)$. Note also that we have the identity
\beeq \label{powerinduction} p( n/2-({n+1})/p-s-\ep)=-(
n/2-(1-s))\ . \eneq

For $k=0$, by \eqref{70-est-Strauss-KeyHigh2} we have
$M_0\le C_0\de$, with $C_0$ a fixed constant. More generally, \eqref{70-est-Strauss-KeyHigh2} implies that
\begin{align}\label{3.18}
M_k\le C_0\de +C_0\sum_{|\alpha|\le 2} \, \Bigl( \,& \bigl\| \,
|x|^{-{n}/2+1-s} Z^\alpha F_p(u^{(k-1)}) \bigr\|_{L^1_t
L^1_{|x|} L^{2}_\omega(\R_+\times \{x: |x|\ge 1\})}
\\
&+\|Z^\alpha F_p(u^{(k-1)}) \|_{L^1_t L^{q_{1-s}'}_x(\R_+\times
\{x: |x|\le 1\})}\Bigr)\,. \notag
\end{align}

Recall that our assumption \eqref{Fp} on the nonlinear term
$F_p$ implies that for small $v$ \beeq\label{small}
\sum_{|\alpha|\le 2}|Z^\alpha F_p(v)|\lesssim |v|^{p-1}
\sum_{|\alpha|\le 2}|Z^\alpha v|+|v|^{p-2} \sum_{|\alpha|\le
1}|Z^\alpha v|^2\,. \eneq
Since the collection $Z$ contains vectors spanning the tangent space
to $S^{n-1}$, by Sobolev embedding we have
$$
\|v(r\cdot)\|_{L^\infty_\om}+\sum_{|\al|\leq 1}\|Z^\al v(r\cdot)\|_{L^4_\omega} \lesssim \sum_{|\alpha|\le 2} \|Z^\alpha
v(r\cdot)\|_{L^{2}_\omega}\,.
$$
Consequently, for fixed $t, r>0$
$$
\sum_{|\alpha|\le 2}\|Z^\alpha F_p(u^{(k-1)}(t,r\cdot) )
\|_{L^{2}_\omega}\lesssim \sum_{|\alpha|\le 2} \|Z^\alpha
u^{(k-1)}(t,r\cdot) \|^p_{L^{2}_\omega}\,.
$$
By \eqref{powerinduction}, the first summand in the right side of
\eqref{3.18} is dominated by $C_1 M_{k-1}^p$ for small
$u^{(k-1)}$.

Since $q'_{1-s}<2<q_s$,
$p>2$ and $n\leq 4$, we can choose $\eta>0$ small enough such that
$p, q_s>2+\eta$ and so $W^{2, 2+\eta}\subset
L^\infty$, $H^{1}\subset L^4$. Thus, for each fixed $t$, we have
\begin{eqnarray*}
&&\sum_{|\alpha|\le 2}\|Z^\alpha F_p(u^{(k-1)}(t,\cdot))
\|_{L^{q'_{1-s}}(x:|x|\le 1)}\\ &\lesssim& \sum_{|\alpha|\le 2}
\|u^{(k-1)}\|^{p-1}_{L^\infty(x:|x|\le 1)} \|Z^\alpha
u^{(k-1)}(t,\cdot)
\|_{L^{q'_{1-s}}(x:|x|\le 1)}\\
&&+\sum_{|\alpha|\le 1} \|u^{(k-1)}\|^{p-2}_{L^\infty(x:|x|\le 1)}
\|Z^\alpha u^{(k-1)}(t,\cdot)
\|^2_{L^{2q'_{1-s}}(x:|x|\le 1)}\\
&\lesssim& \sum_{|\alpha|\le 2} \|u^{(k-1)}\|^{p-1}_{W^{2,
2+\eta}(x:|x|\le 2)} \|Z^\alpha u^{(k-1)}(t,\cdot)
\|_{L^{q_{s}}(x:|x|\le 1)}\\
&&+\sum_{|\alpha|\le 2} \|u^{(k-1)}\|^{p-2}_{W^{2, 2+\eta}(x:|x|\le
2 )} \|Z^\alpha u^{(k-1)}(t,\cdot)
\|^2_{L^{2}(x:|x|\le 2)}\\
&\les& \sum_{|\alpha|\le 2} \|Z^\alpha u^{(k-1)}(t,\cdot)
\|^p_{L^{q_{s}}(x:|x|\le 1)}+\\
&&\sum_{|\be|\le 2} \||x|^{{n}/{2}-({n+1})/{p}-s-\ep} Z^\be
u^{(k-1)}(t,\cdot)\|^{p}_{L^p_{|x|} L^{2+\eta}_\omega(|x|\ge 1)}
\end{eqnarray*}
The second summand in the right side of \eqref{3.18} is thus also
dominated by $C_1 M_{k-1}^p$, and we conclude that $M_k\le C_0\de+2
C_0\,C_1 M_{k-1}^p$. Then
\begin{equation}\label{3.19} M_k\le 2\,C_0\de, \quad k=1,2,3,\dots~.
\end{equation}
for $\de>0$ sufficiently small.
Moreover, the smallness condition of \eqref{small} is verified for
sufficiently small $\delta>0$, since $$\|u^{(k)}\|_{L^\infty_{t,x}}\les
M_k\ .$$

To finish the proof of Theorem \ref{Strauss} we need only to show
that $u^{(k)}$ converges to a solution of the equation
\eqref{eq}. For this it suffices to show that
$$A_k=
\|u^{(k)}-u^{(k-1)}\|_{L^p_t X_{s,\ep,p}}
$$
tends geometrically to zero as $k\to \infty$.  Since
$|F_p(v)-F_p(w)|\lesssim |v-w|(\, |v|^{p-1}+|w|^{p-1}\, )$, the
proof of \eqref{3.19} can be adapted to show that, for small
$\de>0$, there is a uniform constant $C$ so that
$$A_k\le CA_{k-1}(M_{k-1}+M_{k-2})^{p-1},$$
which, by \eqref{3.19}, implies that $A_k\le \tfrac12A_{k-1}$ for
small $\de$.  Since $A_1$ is finite, the claim follows, which
finishes the proof of Theorem \ref{Strauss}.

\subsection{Local Results when $n=3$}
In this subsection we prove Theorem \ref{Strauss2}. Let $2\le
p<p_c=1+\sqrt{2}$ and $n=3$.

Define $s=s_d=1/2-1/p$, and $a$ be the number such that
$$p\left[(n-1)( 1/2- 1/p)-a\right]=1-s- n/2\ ,$$
i.e., $a=-1/{p^2}-({n-1})/({2p})+({n-1})/2$. Since $2\leq
p<p_c$, we have $a\in (0, 1/p)$. By the estimates
\eqref{highorderenergyest}, \eqref{highorderest2} and Duhamel's
principle, we have for $T\ge 1$
\begin{eqnarray}
\sum_{|\al|\le 2} \left (\||x|^{(n-1)( 1/2- 1/p)-a}Z^\al u
\|_{L^p_t L^p_{r}L^2_{\omega}([0,T]\times \{|x|>1\})}+
\|Z^\al u\|_{L^p_t L^{q_s}_x([0,T]\times \{|x|<1\})}\right )&&\nonumber\\
\les T^{ 1/{p}-a+\ep}\sum_{|\al|\le 2} \left (\Vert Z^\al
u_0\Vert_{{\dot H}^{s}}+ \| Z^\al u_1\|_{ \dot H^{s-1}} +\|Z^\al
F\|_{L^1_t \dot
H^{s-1}}\right )&&\nonumber\\
\les T^{ 1/{p}-a+\ep}\sum_{|\al|\le 2} \left (\Vert Z^\al
u_0\Vert_{{\dot H}^{s}}+ \| Z^\al u_1\|_{ \dot H^{s-1}} +\|Z^\al
F\|_{L^1_t X_{1-s,0,\infty}'}\right ) .&&\label{highorderest3}
\end{eqnarray}
Now if we set \begin{multline} M_k = \sum_{|\alpha|\le 2}\left(
\|Z^\alpha u^{(k)} \|_{L^\infty_t\dot H^s }+ \|\partial_t Z^\alpha
u^{(k)}\|_{L^\infty_t \dot H^{s-1}} \right)\\+T^{a- 1/{p}-\ep}
\sum_{|\al|\le 2}\left (\||x|^{(- 1/2-s)/p}Z^\al u \|_{L^p_t
L^p_{r}L^2_{\omega}([0,T]\times \{|x|>1\})}+\|Z^\al u\|_{L^p_t
L^{q_s}_x([0,T]\times \{|x|<1\})}\right),
\end{multline}
 then on the basis of \eqref{highorderenergyest} and
 \eqref{highorderest3}, we
 can use the iteration method (with $\eta=0$) as in Section 4.1 to get the existence
 result for $2\le p<p_c$ and $\rho>2$ in Theorem \ref{Strauss2}.

Heuristically, the lifespan is given when we have
$$M_k\sim \left(T_\de^{1/p-a+\ep} M_k\right)^p\sim \delta \ ,$$
which yields that
$$T_\de\sim \de^{({p(p-1)})/({p^2-2p-1})+\ep'}, ~\forall \ep'>0\ .$$

The case $\rho>1$ can be proved by the same argument in \cite{SW} combined with Theorem \ref{mainest2}.

\vskip 3mm
\noindent{\bf Acknowledgement.} The authors would like to thank the anonymous referee for a careful and thorough reading of the manuscript and a number of helpful comments. In particular, the authors are grateful for the suggestion that leads to an improvement of Theorem \ref{Strauss2}.

\end{document}